\documentclass[english]{amsart}
\usepackage{amssymb,amsmath,amsfonts,amsthm}
\usepackage{stmaryrd}
\usepackage[all,cmtip,poly]{xy}
\usepackage{svn}
\usepackage{enumerate}
 
\SVN $LastChangedDate: 2014-04-24 15:54:03 +0100 (Thu, 24 Apr 2014) $
\SVN $LastChangedRevision: 10 $

\theoremstyle{plain}

\newtheorem{ithm}{Theorem}

\newtheorem{thm}[equation]{Theorem}
\newtheorem{prop}[equation]{Proposition}
\newtheorem{lem}[equation]{Lemma}
\newtheorem{lemma}[equation]{Lemma}
\newtheorem{cor}[equation]{Corollary}

\theoremstyle{definition}
\newtheorem{rem}[equation]{Remark}
\newtheorem{remark}[equation]{Remark}

\newtheorem{defn}[equation]{Definition}
\newtheorem{notn}[equation]{Notation}

\numberwithin{equation}{section}
\numberwithin{figure}{subsection}

\newif\iffinalrun
\iffinalrun
\else
  \usepackage[notref,notcite]{} 
\fi

\iffinalrun
  \newcommand{\need}[1]{}
  \newcommand{\mar}[1]{}
\else
  \newcommand{\need}[1]{{\tiny *** #1}}
  \newcommand{\mar}[1]{\marginpar{\raggedright\tiny #1}}
\fi

\renewcommand{\u}[1]{\underline{#1}}
\newcommand{\A}{\AA}
\newcommand{\C}{\CC}
\newcommand{\F}{\FF}
\newcommand{\N}{\cN}
\newcommand{\Q}{\QQ}
\newcommand{\R}{\RR}
\newcommand{\Z}{\ZZ}

\renewcommand{\O}{\cO}

\newcommand{\m}{\frakm}

\renewcommand{\AA}{{\mathbb A}}

\newcommand{\CC}{{\mathbb C}}

\newcommand{\FF}{{\mathbb F}}

\newcommand{\QQ}{{\mathbb Q}}
\newcommand{\RR}{{\mathbb R}}

\newcommand{\TT}{{\mathbb T}}

\newcommand{\ZZ}{{\mathbb Z}}

\newcommand{\cG}{{\mathcal G}}
\newcommand{\cH}{{\mathcal H}}

\newcommand{\cM}{{\mathcal M}}
\newcommand{\cN}{{\mathcal N}}
\newcommand{\cO}{{\mathcal O}}
\newcommand{\cP}{{\mathcal P}}

\newcommand{\cS}{{\mathcal S}}

\newcommand{\frakm}{\mathfrak{m}}

\newcommand{\frakH}{\mathfrak{H}}

\newcommand{\Fbar}{\overline{\F}}
\newcommand{\Qbar}{\overline{\Q}}
\newcommand{\Zbar}{\overline{\Z}}

\newcommand{\Fp}{\F_p}

\newcommand{\Fpbar}{\Fbar_p}

\newcommand{\Zpbar}{\Zbar_p}

\newcommand{\Qp}{\Q_p}

\newcommand{\Qpbar}{\Qbar_p}

\DeclareMathOperator{\coker}{coker}

\DeclareMathOperator{\End}{End}
\DeclareMathOperator{\Ext}{Ext}
\DeclareMathOperator{\Fil}{Fil}
\DeclareMathOperator{\Gal}{Gal}
\DeclareMathOperator{\GL}{GL}

\DeclareMathOperator{\Hom}{Hom}
\DeclareMathOperator{\im}{im}
\DeclareMathOperator{\Ind}{Ind}

\DeclareMathOperator{\Sym}{Sym}
\DeclareMathOperator{\Trace}{Trace}

\newcommand{\ab}{\mathrm{ab}}
\newcommand{\cris}{\mathrm{cris}}
\newcommand{\dR}{\mathrm{dR}}

\newcommand{\Frob}{\mathrm{Frob}}

\newcommand{\gr}{\mathrm{gr}}

\newcommand{\st}{\mathrm{st}}

\newcommand{\ur}{\mathrm{ur}}

\newcommand{\expl}{\mathrm{expl}}
\newcommand{\ssimp}{\mathrm{ss}}
\newcommand{\modular}{\mathrm{mod}}

\newcommand{\barrho}{\overline{\rho}}
\newcommand{\rhobar}{\overline{\rho}}
\newcommand{\om}{\tilde{\omega}}

\newcommand{\thetabar}{\overline{\theta}}

\newcommand{\la}{\langle}
\newcommand{\ra}{\rangle}

\newcommand{\into}{\hookrightarrow}

\newcommand{\col}{\colon}

\newcommand{\M}{{\cM}}

\newcommand{\Art}{{\operatorname{Art}}}

\newcommand{\BrMod}{\mathrm{BrMod}}

\newcommand{\BrMods}{\BrMod^{K}_{L,k_E}}
\newcommand{\etabar}{\overline\eta}
\DeclareMathOperator{\Nm}{Nm}

\newcommand{\epsilonbar}{\overline{\epsilon}}
\newcommand{\gimelsub}{\Upsilon}
\newcommand{\dalethsub}{\Phi}
\newcommand{\kv}{k}
\begin{document}
\title[Serre weights for locally reducible Galois representations]{Serre weights for locally reducible two-dimensional Galois representations}
\author{Fred Diamond}
\email{fred.diamond@kcl.ac.uk}
\address{Department of Mathematics, King's College London}
\author{David Savitt}
\email{savitt@math.arizona.edu}
\address{Department of Mathematics, University of Arizona}
\keywords{Galois representations, Serre's conjecture,
  Breuil modules}
\subjclass[2010]{11F80}
\thanks{The first author was partially supported by Leverhulme Trust RPG-2012-530;
the second author was partially
  supported by NSF  grant DMS-0901049 and NSF CAREER grant DMS-1054032.}
\maketitle 
 
\begin{abstract}  Let $F$ be a totally real field, and $v$ a place of
  $F$ dividing an odd prime $p$.  We study the weight part of Serre's conjecture for
  continuous totally odd representations $\rhobar:G_F \to \GL_2(\Fpbar)$ that are
  reducible locally at $v$.  Let~$W$ be the set of predicted Serre
  weights for the semisimplification of $\rhobar|_{G_{F_v}}$.  We prove that when
  $\rhobar|_{G_{F_v}}$ is generic, the Serre weights in $W$ for which
  $\rhobar$ is modular are exactly the ones that are predicted
  (assuming that $\rhobar$ is modular).  We also determine precisely
  which subsets of $W$ arise as predicted weights when $\rhobar|_{G_{F_v}}$
  varies with fixed generic semisimplification.
\end{abstract}

\section*{Introduction}

Let $F$ be a totally real field and $\rhobar:G_F \to \GL_2(\Fpbar)$ a continuous
totally odd representation.  Suppose that $\rhobar$ is automorphic in the
sense that it arises as the reduction of a $p$-adic representation of $G_F$
associated to a cuspidal Hilbert modular eigenform, 
or equivalently to a cuspidal holomorphic automorphic representation of
$\GL_2(\A_F)$.

The weight part of Serre's Conjecture in this context was formulated
in increasing generality by Buzzard, Jarvis and one of the authors~\cite{bdj},
Schein~\cite{ScheinRamified} and Barnet--Lamb, Gee and Geraghty~\cite{blggU2} 
(see also \cite{GeePrescribed}).
The structure of the statement is as follows:  Let $v$ be a prime of $F$
dividing $p$, and let $k$ denote its residue field.    A {\em Serre weight}
is then an irreducible representation of $\GL_2(k)$ over~$\Fpbar$.
One can then define what it means for $\rhobar$ to be modular of
a given (Serre) weight, depending {\em a priori} on the choice of a suitable quaternion
algebra over $F$, and we let $W^v_{\modular}(\rhobar)$ denote the set of
weights at $v$ for which $\rhobar$ is modular.
On the other hand one can define a set of weights
$W_{\expl}(\rhobar)$ that depends only on $\rhobar_v = \rhobar|_{G_{F_v}}$, and
the conjecture states that $W^v_{\modular}(\rhobar) = W_{\expl}(\rhobar_v)$.

A series of papers by Gee and coauthors
\cite{geeBDJ, geesavitttotallyramified, GLS11, blggU2, GLS2, GeeKisin} 
proves the following, under mild technical hypotheses on $\rhobar$:
\begin{itemize}
\item $W^v_{\modular}(\rhobar)$ depends only on $\rhobar_v$;
\item $W_{\expl}(\rhobar_v) \subseteq W^v_{\modular}(\rhobar)$;
\item $W_{\expl}(\rhobar_v)  = W^v_{\modular}(\rhobar)$ if $F_v$ is unramified 
or totally ramified over $\Q_p$.
\end{itemize}


In this paper we study the reverse inclusion $W^v_{\modular}(\rhobar)
\subseteq W_{\expl}(\rhobar_v)$ when $\rhobar_v$ is reducible and
$F_v$ is an arbitrary finite extension of $\Qp$ (i.e.\ not necessarily
either unramified or totally ramified). One often refers to this
inclusion as the problem of ``weight elimination,'' since one wishes
to eliminate weights not in $W_{\expl}(\rhobar_v)$ as possible weights for $\rhobar$.

Suppose that $\rhobar_v$ has the form 
$$\left(\begin{array}{cc} \chi_2 & * \\ 0 & \chi_1 \end{array}\right).$$
Then the set $W_{\expl}(\rhobar_v)$ is a subset of $W_{\expl}(\rhobar_v^\ssimp)$ which depends
on the associated extension class $c_{\rhobar_v} \in H^1(G_{F_v}, \Fpbar(\chi_2\chi_1^{-1}))$.
Assume that $\rhobar$ satisfies the hypotheses
of \cite{GeeKisin}, as well as a certain genericity hypothesis (a
condition on $\chi_2 \chi_1^{-1}$; see
Definition~\ref{defn:generic} for a precise statement).  
Then our main global result is the following.
\begin{ithm} \label{ithm:main-global}
$W_\expl(\rhobar_v) = W^v_\modular(\rhobar) \cap W_\expl(\rhobar_v^\ssimp)$.
\end{ithm}
In other words we prove under these hypotheses that weight elimination, and so
also the weight part of Serre's Conjecture, holds
for weights in $W_\expl(\rhobar_v^\ssimp)$.

While the set  $W_{\expl}(\rhobar_v^\ssimp)$ is completely explicit, the 
dependence of $W_{\expl}(\rhobar_v)$
on the extension class is given in terms of the existence of reducible 
crystalline lifts of $\rhobar_v$ with prescribed Hodge--Tate weights.  
In particular it is not clear which subsets of $W_{\expl}(\rhobar_v^\ssimp)$
arise as the extension class $c_{\rhobar_v}$ varies.  Another purpose of the paper is to
address this question, which we resolve in the case where $\rhobar_v$ is generic.
These local results indicate a structure on the sets
$W_{\expl}(\rhobar_v^\ssimp)$. This structure should
reflect properties of a mod $p$ local Langlands correspondence in this
context, in the sense that the set $W_{\expl}(\rhobar_v)$ is expected
to determine the $\GL_2(\cO_{F_v})$-socle of $\pi(\rhobar_v)$, the
$\GL_2(F_v)$-representation associated to $\rhobar_v$ by that correspondence.

To simplify the statement slightly
for the introduction, we assume (in addition to genericity) that the restriction of 
$\chi_2\chi_1^{-1}$ to the inertia subgroup of $G_{F_v}$  is not the cyclotomic
character or its inverse.   In particular this implies that $H^1(G_{F_v},\Fpbar(\chi_2\chi_1^{-1}))$
has dimension $[F_v:\Q_p] = ef$, where $f = [k:\F_p]$ and $e$ is the absolute ramification degree
of~$F_v$.  (We remark that this notation differs slightly from the
notation in the body of the paper, where the ramification degree of
$F_v$ will be $e'$.)  We shall define a partition of $W_{\expl}(\rhobar_v^\ssimp)$
into subsets $W_a$ indexed by the elements $a = (a_0,a_1,\ldots,a_{f-1})$ of
$A = \{0,1,\ldots,e\}^f$, and a subspace $L_a \subseteq
H^1(G_{F_v},\Fpbar(\chi_2\chi_1^{-1}))$ of codimension $\sum_{i=0}^{f-1} a_i$
for each $a \in A$.  We give the set $A$ the usual (product) partial ordering.

Our main local result is the following.
\begin{ithm} \label{ithm:main-local}  Suppose that $\sigma \in W_a$.  Then
$\sigma \in W_\expl(\rhobar_v)$ if and only if $c_{\rhobar_v} \in L_a$.
Moreover there exists $b \in A$ (depending on $\rhobar_v$) such that 
$$W_\expl(\rhobar_v) = \coprod_{a \le b} W_a.$$
\end{ithm}
In other words the weights come in packets, where the packets arise in
a hierarchy compatible with the partial ordering on $A$.  In connection with
the hypothetical mod $p$ local Langlands correspondence mentioned above,
Theorem~\ref{ithm:main-local} is consistent with the possibility that the associated $\GL_2(F_v)$-representation
$\pi(\rhobar_v)$ is equipped with an increasing filtration of length
$[F_v:\Q_p]+1$ such that $\gr^\bullet(\pi(\rhobar_v)) \cong \pi(\rhobar_v^\ssimp)$ and 
$\gr^m(\pi(\rhobar_v))$ has $\GL_2(\cO_{F_v})$-socle consisting
of the weights in the union of the $W_a$ with $\sum_{i=0}^{f-1} a_i = m$
(cf. \cite[Thm.~19.9]{BreuilPaskunas}).

We now briefly indicate how our constructions and proofs proceed.  The set $W_a$ is defined
using the reduction of a certain tamely ramified principal series type $\theta_a$, and the space
$L_a$ is defined using Breuil modules with descent data corresponding to $\theta_a$.
In the first three sections of the paper, we show that the spaces
$L_a$ have the (co-)dimension claimed above, and that they satisfy
$L_a \cap L_{a'} = L_{a''}$ where $a_i'' = \max\{a_i,a_i'\}$.
Section~\ref{sec:extensions-Breuil-modules} contains a general
analysis of the extensions of rank one Breuil modules.  In
Section~\ref{sec:models of principal series type} we define and 
study the
extension spaces $L_a$, and in
Section~\ref{sec:weights-types} we describe our subsets $W_a$ of $W_{\expl}(\rhobar_v^{\mathrm{ss}})$.

  Having done the
local analysis, the strategy for the proving the main results is similar to
that of Gee, Liu and one of the authors~\cite{GLS11} in the totally
ramified case; in particular global
arguments play a role in proving the local results.  More precisely in 
Section~\ref{sec:main-results} we prove the
following three conditions are equivalent for each weight $\mu \in W_a$:
\begin{enumerate}
\item $\mu \in W_\expl(\rhobar_v)$;
\item $\mu \in W^v_\modular(\rhobar)$;
\item $c_{\rhobar_v} \in L_a$.
\end{enumerate}
The implication (1) $\Rightarrow$ (2) is already proved by Gee and Kisin~\cite{GeeKisin},
and (2) $\Rightarrow$ (3) is proved by showing that $\rhobar_v$ has a potentially
Barsotti--Tate lift of type $\theta_a$.  Having now proved that (1) $\Rightarrow$ (3),
one deduces that $L_a$ contains the relevant spaces of extensions with reducible
crystalline lifts; equality follows on comparing dimensions, and this 
gives (3) $\Rightarrow$ (1).

The reason our results are not as definitive as those of
\cite{GLS11} is that in the totally ramified case there is a tight
connection between being modular of some Serre weight and having a
potentially Barsotti-Tate lift of a certain type:\ in
the totally ramified case the reduction mod~$p$ of the principal series type $\theta_a$ has
at most two Jordan--H\"older factors, while in general it can have
many more. 

In fact, when $F_v$ is allowed to be arbitrary some sort of hypothesis
along the lines of genericity is necessary, in the sense that there exist $F_v$, $\chi_1$,
$\chi_2$, and $\mu$ such that the subset of $H^1(G_{F_v},\Fpbar(\chi_2
\chi_1^{-1}))$ corresponding to $\rhobar_v$ with $\mu \in
W_{\expl}(\rhobar_v)$ is not equal to $L_a$ for any choice of
$a$.  We give an example of this phenomenon in Section~\ref{sec:counterexample}.

Finally, we must point out that some time after this paper was written, Gee,
Liu, and the second author \cite{GLS13} announced a proof that
$W_{\expl}(\rhobar_v)  = W^v_{\modular}(\rhobar)$ in general, thus
improving on our Theorem~\ref{ithm:main-global} (by rather different methods).  The arguments in
\cite{GLS13} are entirely local, and depend on an extension to the
ramified case of the $p$-adic Hodge theoretic results proved in
\cite{GLS2} in the unramified case.

\subsection*{Acknowledgments}  This collaboration was initiated during the
Galois Trimester at the Institute Henri Poincar\'e in 2010, and the authors benefited
from further opportunities for discussion during the LMS-EPSRC Durham Symposium (2011)
and workshops at the Institute for Advanced Study (2011) and the
Fields Institute (2012).  Much of the research was carried out during a visit
by DS to King's partly funded by an LMS grant in 2012, and a visit by  FD to the Hausdorff
Institute in 2013 was valuable in completing the work.  The authors are grateful to these
institutions and to the organizers of the above programs for their
support, as well as to the anonymous referee for his or her comments
and suggestions.

\subsection*{Notation and conventions}  If $M$ is a field, we let $G_M$ denote its
absolute Galois group.
If~$M$ is a global field and $v$ is a place of $M$, let $M_v$ denote
the completion of $M$ at $v$.   If
~$M$ is a finite extension of $\Qp$ for some $p$, we let $M_0$
denote the maximal unramified extension of $\Qp$ contained in $M$, and
we write $I_M$
for the inertia subgroup of~$G_M$. 

 Let $p$ be an odd prime number.   Let $K \supseteq L$ be finite extensions of $\Qp$ such that $K/L$ is a tame
Galois extension.  (These may be regarded as fixed, although at certain points in the paper we will make a
specific choice for $K$.)  Assume further that $\pi$ is a uniformiser of
$\O_K$ with the property that $\pi^{e(K/L)} \in L$, where $e(K/L)$ is
the ramification index of the extension $K/L$.   Let $e,f$ and $e',f'$ be the absolute ramification and inertial
degrees of $K$ and $L$ respectively, and denote their residue fields
by $k$ and~$\ell$.  From Section~\ref{sec:comp-extens-class} onwards,
$e(K/L)$ will always be divisible by $p^{f'}-1$, and from Section~\ref{sec:mainsetting}
onwards we will have $f=f'$ and $e(K/L) = p^{f}-1$.
Write $\eta : \Gal(K/L) \to
\cO_K^{\times}$ for the function sending $g \mapsto g(\pi)/\pi$, and
let $\etabar : \Gal(K/L) \to k^{\times}$ be the reduction of $\eta$
modulo the maximal ideal of $\O_K$.

Our representations of $G_L$ will have coefficients in $\Qpbar$,
a fixed algebraic closure of $\Qp$ whose residue field we denote $\Fpbar$.   Let $E$ be a finite
extension of $\Qp$ contained in $\Qpbar$ and containing the image of every
embedding of $K$ into $\Qpbar$. Let $\O_E$ be the ring of integers in
$E$, with uniformiser $\varpi$ and residue field $k_E \subset
\Fpbar$.  Note in particular that there exist $f$
embeddings of $k$ into $k_E$.

We write $\Art_L \col L^\times\to W_L^{\ab}$ for
the isomorphism of local class field theory, normalised so that
uniformisers correspond to geometric Frobenius elements.  For each $\sigma\in \Hom(\ell,\Fpbar)$ we
define the fundamental character $\omega_{\sigma}$ corresponding
to~$\sigma$ to be the composite $$\xymatrix{I_L \ar[r] &
  \O_{L}^{\times}\ar[r] & \ell^{\times}\ar[r]^{\sigma} &
  \Fpbar^{\times},}$$
where the map $I_L \to \O_L^\times$ is induced by the restriction of $\Art_L^{-1}$.
Let $\epsilon$ denote the $p$-adic cyclotomic
character and $\epsilonbar$ the mod~$p$ cyclotomic
character, so that $\prod_{\sigma \in \Hom(\ell,\Fpbar)}
\omega_{\sigma}^{e'} = \epsilonbar$.   We will often identify
characters $I_L \to \Fpbar^{\times}$ with characters $\ell^{\times}
\to \Fpbar^{\times}$ via the Artin map, as above, and similarly for
their Teichm\"uller lifts.

Fix an embedding $\sigma_0 : k \into k_E$, and recursively define
$\sigma_i : k \into k_E$ for all $i \in \Z$ so that $\sigma_{i+1}^p =
\sigma_i$.  We write $\omega_i$ for $\omega_{\sigma_i |_{\ell}}$.
With these normalizations, if $K/L$ is totally ramified of degree $e(K/L) = p^{f'} - 1$
then $\omega_i = (\sigma_i \circ \etabar)|_{I_L}$.  

We normalize Hodge--Tate weights so that all Hodge--Tate weights of
the cyclotomic character are equal to $1$.  (See
Definition~\ref{defn:HTwts} for further discussion of our conventions
regarding Hodge--Tate weights.)

\section{Extensions of Breuil modules}
\label{sec:extensions-Breuil-modules}

In the paper \cite{BreuilAnnals}, Breuil classifies $p$-torsion finite flat group
schemes over~$\cO_K$ in terms of semilinear-algebraic objects that
have come to be known as Breuil modules.  This classification has
proved to be immensely useful, in part because Breuil modules are
often amenable to explicit computation.
In this section we make a careful study of the extensions
between Breuil modules of rank one with coefficients and descent
data.   Many of these results are familiar, but the statements that we
need are somewhat more general than those in
the existing literature  (\textit{cf.}~\cite{BCDT,SavittCompositio,CarusoBDJ,ChengBreuil}).

\subsection{Review of rank one Breuil modules}
\label{sec:rank-one}

We
let $\phi$ denote the endomorphism of $(k \otimes_{\Fp}
k_E)[u]/u^{ep}$ obtained by $k_E$-linearly extending the $p$th power
map on $k[u]/u^{ep}$.  Define an action of $\Gal(K/L)$ on $(k \otimes_{\Fp}
k_E)[u]/u^{ep}$ by the formula $g((a \otimes 1)u^i) = (g(a)\etabar(g)^i \otimes
1)u^i$, extended $k_E$-linearly.

\begin{defn}
  \label{defn:breuil-modules-with-descent-data}
The
category of \emph{Breuil modules with $k_E$-coefficients and generic
  fibre descent
  data from $K$ to $L$}, denoted $\BrMods$, 
is the category whose objects are quadruples $(\mathcal{M},\Fil^1
\mathcal{M},\phi_{1},\{\widehat{g}\})$ where:
\begin{itemize}\item $\mathcal{M}$ is a finitely generated free
  $(k\otimes_{\F_p} k_E)[u]/u^{ep}$-module,
\item $\Fil^1 \M$ is a $(k\otimes_{\F_p} k_E)[u]/u^{ep}$-submodule of $\M$ containing $u^{e}\M$.
\item $\phi_{1}:\Fil^1\M\to\M$ is a $\phi$-semilinear map
  whose image generates $\M$ as a $(k\otimes_{\F_p} k_E)[u]/u^{ep}$-module,
\item the maps $\widehat{g}:\M\to\M$ for each
  $g\in\Gal(K/L)$  are additive bijections that preserve $\Fil^1 \M$, commute with the $\phi_1$-,
  and $k_E$-actions, and satisfy $\widehat{g}_1\circ
  \widehat{g}_2=\widehat{g_1\circ g}_2$ for all
  $g_1,g_2\in\Gal(K/L)$. Furthermore
  $\widehat{1}$ is the identity, and  if $a\in (k\otimes_{\F_{p}} k_E)[u]/u^{ep}$,
  $m\in\M$ then $\widehat{g}(am)=g(a)\widehat{g}(m)$.\end{itemize}
We will usually write $\M$ in place of $(\mathcal{M},\Fil^1
\mathcal{M},\phi_{1},\{\widehat{g}\})$. A morphism $f : \M \to \M'$ in $\BrMods$ is a $(k \otimes_{\Fp}
k_E)[u]/u^{ep}$-module homomorphism with $f(\Fil^1 \M) \subseteq \Fil^1
\M'$ that commutes with $\phi_1$ and the descent data.
\end{defn}

The category $\BrMods$ is equivalent to the category of finite flat
group schemes over $\mathcal{O}_K$ together with a $k_E$-action and descent
data on the generic fibre from $K$ to $L$ (see \cite{BreuilAnnals,SavittRaynaud}).
This equivalence depends on the choice of uniformiser~$\pi$.  The covariant
functor $T_{\st,2}^L$ defined immediately before Lemma~4.9
of~\cite{SavittDuke} associates to each object $\M$ of $\BrMods$ a
$k_E$-representation of $G_L$, which we refer to as
the \emph{generic fibre} of $\M$. 

\begin{notn}\label{notn:idempotent}
We let $e_i \in k \otimes_{\Fp}
k_E$ denote the idempotent satisfying $(x \otimes 1)e_i = (1
\otimes \sigma_i(x))e_i$ for all $x \in k$.  Observe that $\phi(e_i) =
e_{i+1}$.  We adopt the convention that if $m_0,\ldots,m_{f-1}$ are elements of some
$(k\otimes k_E)$-module, then $\u{m}$ denotes the sum
$\sum_{i=0}^{f-1} m_i e_i$, as well as any inferrable variations of
this notation: for instance if $r_0,\ldots,r_{f-1}$ are integers then
$u^{\u{r}}$ denotes $\sum_{i=0}^{f-1} u^{r_i} e_i$.     Conversely for
any element written $\u{a}$, we set $a_i = e_i \u{a}$.  When $\u{a} \in
(k\otimes k_E)[u]/u^{ep}$ we will generally identify $a_i$ with its
preimage in $k_E[u]/u^{ep}$ under the the map $k_E[u]/u^{ep} \simeq
e_i((k\otimes k_E)[u]/u^{ep})$ sending $x \mapsto e_i x$.
\end{notn}

The rank one
objects of $\BrMods$ are classified as follows.

\begin{lem}
  \label{lem:rank-one}
Every rank one object of
$\BrMods$ has the form:
\begin{itemize}
\item $\M = ((k \otimes_{\Fp} k_E)[u]/u^{ep}) \cdot m$,
\item $\Fil^1 \M = u^{\u{r}} \M$,
\item $\phi_1(u^{\u{r}} m) = \u{a}m$ for some $\u{a} \in (k
  \otimes_{\Fp} k_E)^{\times}$, and
\item $\widehat{g}(m) = (\etabar(g)^{\u{c}} \otimes 1)m$ for all $g \in \Gal(K/L)$,
\end{itemize}
where $r_i \in \{0\dots,e\}$ and $c_i \in \Z/(e(K/L))$ are sequences that satisfy $c_{i+1} \equiv p(c_i + r_i) \pmod{e(K/L)}$, and  
the sequences $r_i,c_i,a_i$ are each periodic with period dividing
$f'$.
\end{lem}

\begin{proof}
This is a special case of \cite[Thm.~3.5]{SavittRaynaud}.  In the
notation of that item we have $D = f'$ because of our assumption that
$k$ embeds into $k_E$,  and the periodicity of the sequence $a_i$
is equivalent to $\u{a} \in (\ell \otimes_{\Fp} k_E)^{\times}$.
\end{proof}

\begin{notn}
  We will denote a rank one Breuil module as in
  Lemma~\ref{lem:rank-one} by $\M(\u{r},\u{a},\u{c})$, or else (for
  reasons of typographical aesthetics) by $\M(r,a,c)$.
\end{notn}

We wish to consider maps between rank one Breuil modules, but before
we do so, we note the following elementary lemma.

\begin{lem}
  \label{lem:galois-invariants}
  Let $\Gal(K/L)$ act on $(k\otimes_{\Fp} k_E)[u]/u^{ep}$ by $g \cdot x
  = (\etabar(g)^{\u{w}} \otimes 1) g(x)$, where $g(x)$ denotes the
  usual action and $\{w_i\}$ is a  sequence of integers that is
  periodic with period dividing $f'$.  The $\Gal(K/L)$-invariants of
  this action are the elements $\u{x} \in (k \otimes_{\Fp}
  k_E)[u]/u^{ep}$ such that each nonzero term of $x_i$ has degree
  congruent to $-w_i \pmod{e(K/L)}$, and the sequence $x_i$ is
  periodic with period dividing~$f'$.
\end{lem}

\begin{proof}
  There exists $g \in \Gal(K/L)$ such that  $\etabar(g)=1$ and the
  image of $g$ generates $\Gal(k/\ell)$; since $g(e_i) = e_{i+f'}$,
  the equality  $g\cdot \u{x} = \u{x}$  shows
  that $x_i$ is periodic with period dividing~$f'$.  Consideration of
  the inertia group $I(K/L)$ gives the conditions on the degrees of
  nonzero terms.
\end{proof}

The following
lemma is standard, but its setting is
slightly more general than that of existing statements in the
literature (\emph{cf.}~\cite[Lem.~6.1]{SavittCompositio}, \cite[Prop.~2.5]{ChengBreuil}).

\begin{lem}
  \label{lem:rank-one-maps}
   Let $\M = \M(r,a,c)$ and $\N=\M(s,b,d)$ be rank one Breuil modules
   as above.  
Define $\alpha_i = p(p^{f-1} r_i + \cdots + 
   r_{i+f-1})/(p^f-1)$ and $\beta_i = p(p^{f-1} s_i + \cdots + s_{i+f-1})/(p^f-1)$ for all
   $i$.  There exists a nonzero map $\M \to \N$ if and only if
\begin{itemize}
\item  $\beta_i - \alpha_i \in \Z_{\ge 0}$ for all $i$, 
\item $\beta_i -
   \alpha_i  \equiv c_i - d_i \pmod{e(K/L)}$ for all $i$,
   and
\item $\prod_{i=0}^{f'-1} a_i = \prod_{i=0}^{f'-1} b_i$.
\end{itemize}
\end{lem}

\begin{proof}
A nonzero morphism $\M \to \N$ must have the form $m
\mapsto \u{\delta} u^{\u{z}} n$ for some integers $z_i \ge 0$ and some
$\u{\delta} \in ((k\otimes_{\Fp} k_E)[u]/u^{ep})^{\times}$.  For this map to
preserve
the filtrations, it is necessary and sufficient that $r_i + z_i \ge s_i$ for all $i$.  For
the map to
commute with $\phi_1$ it is necessary and sufficient that $$ \phi(\u{\delta}
u^{(\u{z}+\u{r}-\u{s})}) \u{b} = \u{\delta} \u{a} u^{\u{z}}.$$
It follows from this equation that $\u{\delta} \in (k\otimes_{\Fp}
k_E)^{\times}$, that $z_{i+1} = p(z_i + r_i - s_i)$ for all~$i$, and
that $\phi(\u{\delta})/\u{\delta} = \u{a}/\u{b}$.  The unique solution to the
system of equations for the $z_i$'s is precisely $z_i = \beta_i -
\alpha_i$ for all $i$.  Note that the positivity of $z_{i+1}$
is equivalent to the condition $r_i+z_i \ge s_i$.

For the map to commute with descent data, it is necessary and
sufficient that $g(\u{\delta}) = (\etabar(g)^{\u{c}-\u{z}-\u{d}}\otimes
1) \u{\delta}$ for all $g \in \Gal(K/L)$.  By
Lemma~\ref{lem:galois-invariants}, and recalling that $\u{\delta}$ has no
non-constant terms, this is satisfied if and only
if $z_i \equiv c_i - d_i \pmod{e(K/L)}$ for all $i$ and the
sequence $\delta_i \in k_E$ is periodic with period dividing $f'$.
Finally, it is easy to check that there exists $\u{\delta} \in (k
\otimes_{\Fp} k_E)^{\times}$ with
$\phi(\u{\delta})/\u{\delta} = \u{a}/\u{b}$ and having the necessary periodicity
if and only if  $\prod_{i=0}^{f'-1} a_i = \prod_{i=0}^{f'-1} b_i$.
\end{proof}

\begin{rem}
  \label{rem:if-divisible-remark}
  Suppose that $e(K/L)$ is divisible by $p^{f'}-1$.  By
  \cite[Rem.~3.6]{SavittRaynaud} it is then automatic that the $\alpha_i$
  and $\beta_i$ of the preceding lemma are integers.
  Combining Lemma~\ref{lem:rank-one-maps} with
\cite[Cor.~4.3]{MR2822861}  we see in this case that there exists a
nonzero map $\M \to \N$ if and only if $T_{\st,2}^L(\M) \simeq
T_{\st,2}^L(\N)$ and $\beta_i \ge \alpha_i$ for all $i$.
\end{rem}

We will use the notation $\alpha_i = p(p^{f-1} r_i + \cdots + 
   r_{i+f-1})/(p^f-1)$ throughout the paper, and similarly for $\beta_i$.  Let us write $\Nm(\u{a}) = \prod_{i=0}^{f'-1} a_i \in k_E$.  The following is immediate from (the proof of)
Lemma~\ref{lem:rank-one}.

\begin{cor}
  \label{cor:isomorphism}
  We have $\M(r,a,c) \simeq \M(r',a',c')$ if and only if $r_i = r_i'$
  for all $i$, $c_i = c'_i$ for all $i$, and $\Nm(\u{a}) = \Nm(\u{a}')$.
\end{cor}

The following proposition is again standard, but slightly more general
than the versions in the existing literature
(\cite[Prop.~2.6]{ChengBreuil}, \cite[Prop.~5.6]{CarusoBDJ}).

\begin{prop}
  \label{prop:max-model}
   Let $\M = \M(r,a,c)$ and $\N=\M(s,b,d)$ be rank one Breuil modules
   as above.  There exists a rank one Breuil module $\cP$ and a pair of
   nonzero maps $\M \to \cP$ and $\N \to \cP$ if and only if
\begin{itemize}
\item  $\beta_i - \alpha_i \in \Z$ for all $i$, 
\item $\beta_i -
   \alpha_i  \equiv c_i - d_i \pmod{e(K/L)}$ for all $i$,
   and
\item $\prod_{i=0}^{f'-1} a_i = \prod_{i=0}^{f'-1} b_i$.
\end{itemize}
In fact it is possible to take $\cP = \M(t,a,v)$ such that if $\gamma_i
= p(p^{f-1} t_i + \cdots + t_{i+f-1})/(p^f-1)$ then $\gamma_i =
\max(\alpha_i,\beta_i)$.
\end{prop}

\begin{proof}
It follows directly from Lemma~\ref{lem:rank-one-maps} that the
listed conditions are necessary.  For sufficiency, we follow the
argument of \cite[Prop.~5.6]{CarusoBDJ}.  Define $\gamma_i =
\max(\alpha_i,\beta_i)$, $n_i = \frac{1}{p} \max(0,\beta_i-\alpha_i)$,
$t_i = r_i + pn_i - n_{i+1}$, and $v_i \equiv c_i + (\alpha_i -
\gamma_i) \pmod{e(K/L)}$.  Observe that $n_i$ and $\alpha_i-\gamma_i$
are integers, so that $t_i$ is an integer and $v_i$ is well-defined.
An argument identical to the one at \emph{loc.~cit.} shows that $t_i
\in [0,e]$, and easy calculations show that $\gamma_i
= p(p^{f-1} t_i + \cdots + t_{i+f-1})/(p^f-1)$ and $v_{i+1} \equiv
p(v_i + t_i) \pmod{e(K/L)}$.  Thus $\cP = \M(t,a,v)$ is a
Breuil module with the property given in the last sentence of the proposition, and two applications of Lemma~\ref{lem:rank-one-maps}
show that there exist nonzero maps $\M \to \cP$ and $\N \to \cP$.
(For the latter, note that $\gamma_i - \alpha_i \equiv c_i - v_i
\pmod{e(K/L)}$, and together with our other hypotheses this
implies that $\gamma_i - \beta_i \equiv d_i - v_i \pmod{e(K/L)}$.)
\end{proof}

\begin{cor}
  \label{cor:generic-fibre-isom}
  The conditions in Proposition~\ref{prop:max-model} give
  necessary and sufficient conditions that $T_{\st,2}^L(\M) \simeq
T_{\st,2}^L(\N)$.
\end{cor}

\begin{proof}
Suppose that there exists $\cP$ as in
Proposition~\ref{prop:max-model}.  Since the kernels of the maps produced by
Lemma~\ref{lem:rank-one-maps} do not contain any free
$k[u]/u^{ep}$-submodules, it follows from
\cite[Prop.~8.3]{SavittCompositio} that they induce isomorphisms $T_{\st,2}^L(\M) \simeq
T_{\st,2}^L(\cP)$ and $T_{\st,2}^L(\N) \simeq
T_{\st,2}^L(\cP)$.

Conversely, suppose $T_{\st,2}^L(\M) \simeq
T_{\st,2}^L(\N)$.  Let $\M,\N$ correspond to the rank one $k_E$-vector
space 
schemes $\cG,\cH$ with generic fibre descent data.  By a theorem of
Raynaud \cite[Prop.~2.2.2, Cor~2.2.3]{Raynaud} there exists a maximal rank one $k_E$-vector
space scheme $\cG'$ with nonzero maps $\cG' \to \cG$, $\cG' \to \cH$,
and $\cG'$ obtains generic fibre descent data by a scheme-theoretic closure argument as in
\cite[Prop.~4.1.3]{BCDT}.  Then we can take $\cP$ to be the Breuil
module corresponding to $\cG'$.
\end{proof}

\subsection{Extensions of rank one Breuil modules}
\label{sec:extensions-rank-one-1}

We now describe the extensions between the rank one objects
of $\BrMods$.
 The main result is analogous to
 \cite[Lem.~5.2.2]{BCDT}, 
\cite[Thm.~7.5]{SavittCompositio} and
\cite[Thm.~3.9]{ChengBreuil}, and since the proof is
substantively the same as the proofs given at those references, we will omit some details
of the argument.

\begin{thm} \label{thm:extensions} Let $\M,\N$ be rank one Breuil modules, with notation as
  in~Section~\ref{sec:rank-one}. 
 Each extension of $\M$ by $\N$ is isomorphic to precisely
  one of the form
\begin{itemize}
\item $\cP =   ((k \otimes_{\Fp} k_E)[u]/u^{ep}) \cdot m +  ((k
  \otimes_{\Fp} k_E)[u]/u^{ep}) \cdot n$,
\item $\Fil^1 \cP = \la u^{\u{s}}n, u^{\u{r}} m + \u{h} n \ra$,
\item $\phi_1( u^{\u{s}} n) = \u{b}n$ and $\phi_1(u^{\u{r}} m + \u{h}
  n ) = \u{a} m$,
\item $\widehat{g}(n) = (\etabar(g)^{\u{d}} \otimes 1)n$ and
  $\widehat{g}(m) = (\etabar(g)^{\u{c}} \otimes 1) m$ for all $g \in \Gal(K/L)$,
\end{itemize}
in which each $h_i \in k_E[u]/u^{ep}$ is a polynomial such that:
\begin{itemize}
\item $h_i$ is divisible by $u^{r_i+s_i - e}$,
\item the sequence $h_i$ is periodic with period dividing $f'$,
\item each nonzero term of $h_i$ has degree congruent to $r_i +
  c_i - d_i \pmod{e(K/L)}$, and
\item $\deg(h_i) < s_i$, 
except that when there exists a
nonzero morphism $\M \to \N$, the polynomials $h_i$ for $f' \mid i$
may also have a term of degree $r_0 + \beta_0
- \alpha_0$ in common.
\end{itemize}

In particular the dimension of $\Ext^1(\M,\N)$ is given by the formula
$$ \delta + \sum_{i=0}^{f'-1} \#\left\{ j \in \left[\max(0,r_i+s_i-e),s_i\right)  : j
  \equiv r_i +c_i - d_i \pmod{e(K/L)} \right\}$$ where $\delta = 1$
if there exists a map $\M \to \N$ and $\delta =0 $ otherwise.
\end{thm}

\begin{proof}
Let $\cP$ be any extension of $\M$ by $\N$.  Then $\Fil^1 \cP = \la
u^{\u{s}}n,u^{\u{r}}m + \u{h}n \ra$ for some  $\u{h}$ and some lift
$m$ of the given generator of $\M$, and $\phi(u^{\u{r}}m + \u{h}n) =
\u{a}m + \delta n$ for some $\delta$.  Replacing $m$ with $m +
\delta \u{a}^{-1} n$ and suitably altering $\u{h}$ shows that we can
take $\delta = 0$.  The condition that each $h_i$ is divisible by
$u^{r_i+s_i-e}$ is necessary and sufficient to ensure that $\Fil^1 \cP
\supset u^e \cP$, so that the first three conditions given in the 
statement of the theorem define a Breuil module (without descent
data).  One checks straightforwardly that replacing $m$ with $m +
\u{a}^{-1} \phi(\u{t}) n$ for any $\u{t} \in (k\otimes_{\Fp} k_E)[u]/u^{ep}$ preserves the shape of $\cP$ while replacing
$\u{h}$ with $h - u^{\u{r}} (\u{ba^{-1}}) \phi(\u{t}) + u^{\u{s}} \u{t}$, and
that these are precisely the changes of $m$ that preserve the shape of
$\cP$.

Now the descent data on $\cP$ must have the shape $$\widehat{g}(m) =
(\etabar(g)^{\u{c}} \otimes 1)m + A_g n$$ for some collection of
elements $A_g \in
(k \otimes_{\Fp} k_E)[u]/u^{ep}$.   The condition that $\widehat{hg} =
\widehat{h} \circ \widehat{g}$, evaluated at $m$, implies that the
function $g \mapsto (\etabar(g)^{-\u{c}}\otimes 1) A_g$ is a cocycle in
the cohomology group $H^1(\Gal(K/L),(k\otimes
k_E)[u]/u^{ep})$  in which the action of $\Gal(K/L)$ on $(k\otimes
k_E)[u]/u^{ep}$ is given by $g \cdot x = (\etabar(g)^{\u{d}-\u{c}}
\otimes 1) g(x)$, where $g(x)$ is the usual action.  This cohomology
group is trivial since $\Gal(K/L)$ is assumed to have order prime to
$p$, so that $(\etabar(g)^{-\u{c}} \otimes 1)A_g$ is the coboundary of
some element $v$.  The relation
$\phi_1 \circ \widehat{g} = \widehat{g} \circ \phi_1$ applied to
$u^{\u{r}}m + \u{h}n$ implies that $A_g n$ lies in the image of
$\phi_1$, so that all nonzero terms in each $A_g$ have degree
divisible by $p$; it follows that we can take $v$ to have the same property.
One computes that replacing~$m$ with $m +
\u{a}^{-1} \phi(\u{t}) n$ changes $(\etabar(g)^{-\u{c}} \otimes 1)A_g$ by
the coboundary of $\u{a}^{-1} \phi(\u{t})$, and choosing~$\u{t}$ so that
$\u{a}^{-1} \phi(\u{t}) = -v$ allows us to take $A_g = 0$ for all $g$.

Thus our extension $\cP$ has the shape as in the theorem, and it
remains to investigate the possibilities for $\u{h}$.  In order that
the given shape of $\cP$ actually defines a Breuil module with descent
data, it is necessary and sufficient that $u^{r_i + s_i - e}$ divides
each $h_i$, and that the relation $\phi_1 \circ \widehat{g} =
\widehat{g} \circ \phi_1$ is well-defined and satisfied when evaluated
at $u^{\u{r}}m + \u{h} n$.  A direct calculation shows that the latter
condition is equivalent to the condition that $u^{e + \u{s}}$
divides $$(\etabar^{\u{d}}(g) \otimes 1) g(\u{h}) -
(\etabar^{\u{r}+\u{c}}(g) \otimes 1) \u{h}$$ for all $g \in \Gal(K/L)$,
or equivalently that the remainder of $\u{h}$ upon division by
$u^{e+\u{s}}$
is invariant under the action of Lemma~\ref{lem:galois-invariants}
with $\u{w} = \u{r}+\u{c}-\u{d}$.  From that lemma, we deduce that
any term of $h_i$ of degree $D < e+s_i$ must satisfy $D \equiv r_i +
c_i - d_i \pmod{e(K/L)}$, and that such terms occur periodically
  with period dividing $f'$.  Let  $V \subseteq (k \otimes
  k_E)[u]/u^{ep}$ be the space of elements $\u{h}$ satisfying the
  conditions in the previous sentence and with each $h_i$ divisible by $u^{\max(0,r_i+s_i-e)}$.

Now let us examine the changes-of-variable $m \leadsto m + \u{a}^{-1} \phi(\u{t}) n$
that preserve the shape of $\cP$ (but may change $\u{h}$).  From
the argument two paragraphs earlier, we see that such a change of
variables preserves the shape of the descent data precisely when the
coboundary of $\u{a}^{-1} \phi(\u{t})$ is trivial, or in other words
precisely when $\phi(\u{t}) = g \cdot \phi(\u{t})$ under the
$\Gal(K/L)$-action of that paragraph.  Thus $\u{t}$ may have arbitrary
terms of degree at least $e$ (since $\phi(u^{e}) = 0$), while by
Lemma~\ref{lem:galois-invariants} the
nonzero terms of
$t_i$ of degree $D <e$ must have $D \equiv
p^{-1}(c_{i+1}-d_{i+1}) \pmod{e(K/L)}$, and these terms must occur
periodically with period dividing $f'$.  We say that a choice of $\u{t}$ with 
these properties is \emph{allowable}.

Recall from the beginning of the proof that  replacing $m$ with $m +
\u{a}^{-1} \phi(\u{t}) n$ has the effect of replacing $\u{h}$ with
$\u{h}' = \u{h} - u^{\u{r}} (\u{ba^{-1}}) \phi(\u{t}) + u^{\u{s}}
\u{t}$. 
Let $U \subseteq (k \otimes k_E)[u]/u^{ep}$ be the space of allowable
choices of $\u{t}$, and $\gimelsub : U \to V$ the map that sends $\u{t}$
to $u^{\u{r}}(\u{ba^{-1}}) \phi(\u{t}) - u^{\u{s}} \u{t}$.  The above
discussion shows that $\Ext^1(\M,\N) \simeq \coker(\gimelsub)$.  We use
this isomorphism to compute $\dim_{k_E} \Ext^1(\M,\N)$.  Let $y_i = \#\left\{ j \in \left[\max(0,r_i+s_i-e),s_i\right)  : j
  \equiv r_i +c_i - d_i \pmod{e(K/L)} \right\}$.  One calculates
directly from their definitions that
$$\dim_{k_E} U = e'f' + ef(p-1), \qquad \dim_{k_E} V = e'f' + ef(p-1) + \sum_{i=0}^{f-1} y_i -
\sum_{i=0}^{f-1} s_i. $$ 

Suppose that $\u{t} \in
\ker(\gimelsub)$, i.e.~that $u^{\u{r}}(\u{ba^{-1}}) \phi(\u{t}) =
u^{\u{s}} \u{t}$.  Observe (e.g.~by comparing with the proof of
Lemma~\ref{lem:rank-one-maps})  that this is precisely the condition
required for the map $\M \to \N$ defined by $m \mapsto \phi(\u{t})n$
to be a map of Breuil modules.  If there are no such nonzero maps
(i.e.~if $\delta = 0$, with $\delta$ as in the statement of the Theorem), then
$\ker(\gimelsub) = \{ \u{t} \in U \, : \, u^{\u{s}} t = 0 \}$ and so
$\ker(\gimelsub)$ has dimension $\sum_i s_i$.    If instead there exists
a nonzero map $\M \to \N$ (i.e.~if $\delta = 1$), then since that map
must be unique up to scaling, we see that $u^{\u{s}} \u{t}$ is
unique up to scaling and $\ker(\gimelsub)$ has dimension $1 + \sum_i
s_i$.  In either case $\dim_{k_E} \ker(\gimelsub) = \delta + \sum_i s_i$.
Finally we calculate that $\coker(\gimelsub)$ has dimension 
$$\dim_{k_E} V - \dim_{k_E} U + \dim_{k_E} \ker(\gimelsub) =  \delta +
\sum_{i=0}^{f-1} y_i.$$

Now let $W' \subseteq V$ be the space of elements $\u{h}$ satisfying
the conditions given in the statement of the theorem, and $W
\subseteq W'$ the subspace of elements $\u{h}$ for which the coefficient of degree $r_0 +
\beta_0 - \alpha_0$ in $h_0$ is zero.  (Thus $W \subsetneq W'$ if and only if
$\delta = 1$, in which case $W'/W$ has $k_E$-dimension $1$.)
 It is easy to verify
that $\dim_{k_E} W' = \delta + \sum_i y_i$, and so to complete the
proof of the Theorem it suffices to show that $W' \cap \im(\gimelsub) =
0$.   When $\delta = 1$, a straightforward computation (using the fact
that $\Nm(\u{a}) = \Nm(\u{b})$ in this case) shows that if $\u{h} \in
\im(\gimelsub)$ then the coefficients $\xi_i$ of degree $r_i + \beta_i - \alpha_i$
in $h_i$ for $i = 0,\ldots,f'-1$ satisfy the linear relation
$\sum_{i=0}^{f-1} (a_0 \cdots a_i)(b_0 \cdots b_i)^{-1} \xi_i = 0$.
If in addition we have $\u{h} \in W'$  (so that $\xi_i = 0$ for $i
\not\equiv 0 \pmod{f'}$) then $\xi_0 = 0$ and $\u{h} \in W$.   We are
therefore reduced in all cases to showing that $W \cap \im(\gimelsub) = 0$.

Let $\pi_W : V \to W$ be the projection map that kills each term of $h_i$
of degree at least $s_i$.
Observe that we may write
$$ \gimelsub(\u{t}) = u^{\u{s}} \dalethsub(\u{t}) + \pi_W(\gimelsub(\u{t}))$$
where $\dalethsub(\u{t}) = u^{\u{r}-\u{s}}(\u{ba^{-1}}) \phi(\u{t}) -
\u{t}$, with terms of negative degree in $\dalethsub(\u{t})$
understood to be zero.  To finish the argument we must show that if $u^{\u{s}} \dalethsub(\u{t}) = 0$ then $\pi_W(\gimelsub(\u{t})) = 0$.

Observe that the defining formula for $\dalethsub$
also gives a well-defined 
map $\overline{\dalethsub} \in \End((k\otimes k_E)[u]/u^e)$.
 Fix an integer $v_i \in [0,e)$ and recursively define
$v_j = (r_j - s_j) + p v_{j-1}$ for $j > i$.  Since $u^{v_j} e_j$ and
$(b_{j+1}/a_{j+1}) u^{v_{j+1}} e_{j+1}$ are congruent modulo the image of
$\overline{\dalethsub}$ (where the $e_j$'s are the idempotents defined
in~\ref{notn:idempotent}), it follows that $u^{v_j} e_j \in
\im(\overline{\dalethsub})$ except possibly if the sequence
$\{v_j\}$ lies entirely within the interval $[0,e)$.  In the latter
case the sequence $\{v_j\}$ must be periodic, indeed with period dividing $f'$,
and one computes that $v_j = p^{-1}(\beta_{j+1} - \alpha_{j+1})$ for
all $j$.  Then one checks that $u^{v_j} e_j$ and $\Nm(\u{ba^{-1}}) u^{v_j} e_j$ are
congruent modulo $\im(\overline{\dalethsub})$; so unless $\Nm(\u{a})
= \Nm(\u{b})$ we  again have $u^{v_j} e_j
\in \im(\overline{\dalethsub})$.   We conclude that $\overline{\dalethsub}$ is surjective
(hence bijective) unless $\Nm(\u{a}) = \Nm(\u{b})$ and $p^{-1}(\beta_i
- \alpha_i) \in \{0,\ldots,e-1\}$ for all~$i$, in which case the image of
$\overline{\dalethsub}$ has codimension at most~$1$; and in all cases we conclude that
$\ker(\overline{\dalethsub}) = \ker(\gimelsub') + u^e (k \otimes k_E)[u]/u^{ep}$, where $\gimelsub'$ is the
endomorphism of $(k\otimes k_E)[u]/u^{ep}$ given by the same  defining
formula as $\gimelsub$.  

Now if $u^s \dalethsub(\u{t}) = 0$ then $\overline{\dalethsub}(\u{t}) = 0$,
so $\u{t} \in \ker(\gimelsub') + u^e(k \otimes k_E)[u]/u^{ep}$; finally
$$\gimelsub(\u{t}) = \gimelsub'(\u{t}) \in \gimelsub'(u^e(k\otimes
k_E)[u]/u^{ep}) \subseteq u^e(k\otimes k_E)[u]/u^{ep},$$ 
and it follows that $\pi_W(\gimelsub(\u{t})) = 0$.
\end{proof}

\begin{rem}\label{rem:extension-notation}
 We have seen in the proof of Theorem~\ref{thm:extensions} that $\cP$
 as in the first set of bullet points of Theorem~\ref{thm:extensions}
 is a well-defined Breuil module provided
 that
 \begin{itemize}
 \item $h_i$ is divisible by $u^{r_i + s_i - e}$,
 \item nonzero terms of $h_i$ of degree less than $e+s_i$ have
   degree congruent to $r_i + c_i - d_i \pmod{e(K/L)}$, and occur
   periodically (for $i$) with
   period dividing~$f'$.
 \end{itemize}
We will denote this Breuil module by $\cP(r,a,c;s,b,d;h)$.
\end{rem}

\subsection{Comparison of extension classes}
\label{sec:comp-extens-class}

We assume for the remainder of this paper that $e(K/L)$ is divisible
by $p^{f'}-1$, so that in particular Remark~\ref{rem:if-divisible-remark} is
in force.  We fix characters $\chi_1,\chi_2 : G_L \to k_E^{\times}$
and suppose that $\M = \M(r,a,c)$ and $\N=\M(s,b,d)$ are rank one Breuil modules whose generic
fibres are $\chi_1,\chi_2$ respectively.  The following lemma is \cite[Cor.~4.3]{MR2822861}.

\begin{lem}
 \label{lem:tst-evaluation}
  Set $\M = \M(r,a,c)$ and write $\lambda = \Nm(\u{a})^{-1}$.  Then
  $T_{\st,2}^K(\M) = (\sigma_i \circ \etabar^{c_i + \alpha_i}) \cdot
\ur_{\lambda}$, where $\ur_{\lambda}$ is the unramified character of $G_L$
sending an arithmetic Frobenius element to $\lambda$.
\end{lem}

The character $\chi_1$ and the sequence of $r_i$'s determine $\M$ up to
isomorphism (\emph{cf.}~Corollaries~\ref{cor:isomorphism} and~\ref{cor:generic-fibre-isom}), and
similarly for $\N$; moreover, one checks from Lemmas~\ref{lem:rank-one} and~\ref{lem:tst-evaluation} that given $\chi_1$ and $r_0,\ldots,r_{f-1}$ such that $\alpha_i \in
\Z$ for some (hence all) $i$, there exists $\M(r,a,c)$
with generic fibre $\chi_1$. In the remainder of this section, we compare
extension classes in $\Ext^1(\M,\N)$ with extension classes in
$\Ext^1(\M',\N')$ where $\M',\N'$ are certain other Breuil modules
with the same generic fibres as $\M,\N$ respectively; our treatment
follows the treatment of the case $f=1$ in Section~5.2 of \cite{GLS11}.

\begin{prop}
  \label{prop:comparison-with-minimax}
The Breuil module $\cP = \cP(r,a,c;s,b,d;h)$ has the same generic fibre as
$\cP^\dag = \cP(0,a,c^\dag;e,b,d^\dag;u^{\u{\delta}}h)$ 
where 
\begin{itemize}
\item $c^\dag_i = c_i + \alpha_i$, 
\item $d^\dag_i = d_i + \beta_i -
ep/(p-1)$, and 
\item $\delta_i = ep/(p-1) - \beta_i + \alpha_i - r_i$.
\end{itemize}
\end{prop}

\begin{proof}
Consider the Breuil module $\cP^\ddag = \cP(r,a,c;e,b,d^\dag;u^{\u{\delta}^{\ddag}} h)$
where $\delta^\ddag_i = ep/(p-1) - \beta_i$.  It is elementary
from Remark~\ref{rem:extension-notation} that both $\cP^\dag$ and $\cP^\ddag$
are well-defined (note that $\delta_i,\delta^\ddag_i \ge 0$ and that
$e(K/L)$ divides $e$); the key point of the calculation is that $\beta_i - s_i =
\beta_{i+1}/p \le e/(p-1)$, whence $e-s_i \le
ep/(p-1) - \beta_i$.  Let $m^\dag,n^\dag$ and $m^\ddag,n^\ddag$ denote the standard basis
elements for $\cP^\dag,\cP^\ddag$ respectively.  One checks without difficulty that
there is a map $f^\ddag : \cP \to \cP^\ddag$ sending
$$ m \mapsto m^\ddag,\qquad n \mapsto u^{ep/(p-1) - \u{\beta}} n^\ddag$$
as well as a map $f^\dag : \cP^\dag \to \cP^\ddag$ sending
$$ m^\dag \mapsto u^{\u{\alpha}} m^\ddag,\qquad n^\dag \mapsto n^\ddag.$$
Since $\ker(f^\dag),\ker(f^\ddag)$ do not contain any free
$k[u]/u^{ep}$-submodules, it follows from
\cite[Prop.~8.3]{SavittCompositio} that $T_{\st,2}^L(f^\dag)$ and
$T_{\st,2}^L(f^\ddag)$ are isomorphisms.
\end{proof}

Note that while the extension classes
$\Ext^1_{k_E[G_L]}(\chi_1,\chi_2)$ realized by $\cP$ and~$\cP^\dag$
in Proposition~\ref{prop:comparison-with-minimax} may not coincide,
they differ by at most multiplication by a $k_E$-scalar, since the
maps $f^\dag$ and $f^\ddag$ induce $k_E$-isomorphisms on the one-dimensional
sub and quotient characters.

\begin{defn}
  \label{defn:flat-extension-space}
  Let $L(\M,\N) \subseteq \Ext^1_{k_E[G_L]}(\chi_1,\chi_2)$
  denote the subspace consisting of extension classes of the form
  $T^L_{\st,2}(\cP)$ for $\cP \in \Ext^1(\M,\N)$.
\end{defn}

The following proposition gives a criterion for one space of
extensions $L(\M,\N)$ to be contained in another.

\begin{prop}
  \label{prop:containment}
  Suppose that $\M=\M(r,a,c)$ and $\M'=\M(r',a',c')$ have generic
  fibre $\chi_1$ while $\N=\M(s,b,d)$ and $\N'=\M(s',b',d')$
  have generic fibre $\chi_2$.   If there exist nonzero maps $\M \to \M'$ and $\N' \to \N$ then $L(\M',\N')
  \subseteq L(\M,\N)$.
\end{prop}

\begin{proof}
We show more generally that the conclusion holds provided that $$\max(\alpha_{i+1}/p - \beta_{i},\alpha_{i}-\beta_{i+1}/p - e) \le 
\max(\alpha'_{i+1}/p - \beta'_{i},\alpha'_{i}-\beta'_{i+1}/p - e)$$
for all $i$.  (This inequality is easily checked when there exist maps $\M \to \M'$ and $\N' \to \N$, because $\alpha_i \le
 \alpha_i'$ and $\beta'_i \le \beta_i$ for all $i$ in this case.)

By Corollaries~\ref{cor:isomorphism} and~\ref{cor:generic-fibre-isom}
we may suppose without loss of generality that $a=a'$ and $b=b'$.  Suppose that $\cP' = \cP(r',a,c';s',b,d';h')$.  The given inequality is equivalent to
$$(\beta_i-\alpha_i+r_i) - (\beta'_i-\alpha'_i+r'_i) +
\max(0,r'_i+s'_i-e) \ge \max(0,r_i+s_i-e)$$
which is precisely the condition that is required to make the
assigments $\u{h} = u^{(\u{\beta}-\u{\beta}')-(\u{\alpha}-\u{\alpha}') +
  (\u{r}-\u{r}')} \u{h}' $
and $ \cP = \cP(r,a,c;s,b,d;h)$ well-defined.  Then $\cP$ and~$\cP'$ both have the same generic fibre as the extension $\cP^{\dag}$
of Proposition~\ref{prop:comparison-with-minimax}, and so the generic
fibre of $\cP'$ is also in $L(\M,\N)$.
\end{proof}

We remark that Proposition~\ref{prop:containment} should also follow
from a scheme-theoretic closure argument, but we give the above
argument for the sake of expedience (we will need
Proposition~\ref{prop:comparison-with-minimax} again in
Section~\ref{sec:mainsetting}).

\section{Models of principal series type}
\label{sec:models of principal series type}

We retain the notation and setting of the previous section; in
particular recall that we have a running assumption that $e(K/L)$ is
divisible by $p^{f'}-1$.   Fix a pair of characters $\chi_1,\chi_2 :
G_L \to k_E^{\times}$.

Recall that a two-dimensional \emph{Galois type} is (the isomorphism
class of) a representation
$\tau : I_L \to \GL_2(\Zpbar)$ that extends to a representation of $G_L$ and
whose kernel is open.  We say that $\tau$ is a \emph{principal series
  type} if $\tau \simeq \lambda \oplus \lambda'$ where
$\lambda,\lambda'$ both extend to representations of $G_L$. 

In this section we use the results of
Section~\ref{sec:extensions-Breuil-modules}  to associate to the triple
$(\chi_1,\chi_2,\tau)$ a subspace $L(\chi_1,\chi_2,\tau) \subseteq
\Ext^1_{k_E[G_L]}(\chi_1,\chi_2)$.  We will see that $L(\chi_1,\chi_2,\tau)$
contains every extension of $\chi_1$ by
$\chi_2$ that arises as the reduction mod $p$ of a potentially
Barsotti-Tate representation of type $\tau$; in fact we will think of
$L(\chi_1,\chi_2,\tau)$ as a finite flat avatar for
the collection of such extensions.  In
Section~\ref{sec:models} we define the set $L(\chi_1,\chi_2,\tau)$ and
prove that it is a vector space (provided that it is nonempty).  In Section~\ref{sec:mainsetting} we
restrict to the main local setting of our paper and study the spaces
$L(\chi_1,\chi_2,\tau)$ in detail in that setting; for instance we compute the
dimension of these spaces in many cases.

\subsection{Maximal and minimal models of type $\tau$} \label{sec:models}

Raynaud \cite{Raynaud} shows that if one fixes a finite flat
$p$-torsion group scheme $G$ over $K$, then the set of finite flat group
schemes over $\cO_K$ with generic fibre $G$ has the structure of a
lattice; in particular it possesses maximal and minimal elements. This 
has proved to be a valuable  observation, and variants of it have recurred in
numerous contexts (see \cite[Lem.~4.1.2]{BCDT}, \cite[\S8]{SavittCompositio},
\cite[\S3.3]{MR2543474}, and \cite[\S5.3]{GLS13} to name a few).  Let
$\tau$ be a principal series type. In
this subsection we introduce the notion of a model of type $\tau$ (see
Definition~\ref{defn:model-of-type-tau} below) and prove the existence
of maximal and minimal models of type $\tau$. 

\begin{defn}
  \label{defn:chi-dual}
  Write $\chi = \chi_1 \chi_2$.  If $\M =
  \M(r,a,c)$ has generic fibre $\chi_1$, define the \emph{$\chi$-dual} of $\M$ to be the unique Breuil module
  $\M^{\vee}_{\chi} = \M(s,b,d)$ with generic fibre $\chi_2$ such that
  $r_i + s_i = e$ for all $i$.  The existence of $\M^{\vee}_{\chi}$
  is implied by the paragraph following
  Lemma~\ref{lem:tst-evaluation}.  
\end{defn}

If $\tau \simeq \lambda \oplus \lambda'$ is a principal series type, we let
$\overline{\lambda},\overline{\lambda}'$ denote the reductions
of $\lambda,\lambda'$ modulo the maximal ideal of $\Zpbar$; we  will
usually abuse notation and write $\lambda,\lambda'$ where we mean
$\overline{\lambda},\overline{\lambda}'$.

\begin{defn}
  \label{defn:model-of-type-tau}
  Let $\tau \simeq \lambda \oplus \lambda'$ be a 
  principal
  series 
  type. 
  We say that $\M(r,a,c)$ is a \emph{model of type $\tau$} if
  $\sigma_i \circ \etabar^{c_i} \in \{\lambda,\lambda'\}$ for all
  $i$. 
  Note that if $(\chi_1 \chi_2)|_{I_{G_L}} = \lambda \lambda'
  \epsilonbar$ and $\M(r,a,c)$ is a model of type $\tau$ with generic
  fibre $\chi_1$, then its $\chi$-dual $\M^{\vee}_{\chi} = \M(s,b,d)$ is a
  model of type $\tau$ with generic fibre $\chi_2$, and moreover
  $\{\sigma_i\circ\etabar^{c_i} , \sigma_i\circ\etabar^{d_i}\} =
  \{\lambda,\lambda'\}$ for all $i$.
\end{defn}

\begin{defn}
  \label{defn:lflat}
  We define $$ L(\chi_1,\chi_2,\tau) = \cup_{\M,\N} L(\M,\N) $$
 as $\M,\N$ range over all pairs of models of type $\tau$ with generic fibre
 $\chi_1,\chi_2$ respectively, and such that  $\{\sigma_i\circ\etabar^{c_i} , \sigma_i\circ\etabar^{d_i}\} =
  \{\lambda,\lambda'\}$ for all $i$.
\end{defn}

It follows, for instance from~\cite[Cor.~5.2]{MR2822861}, that
$L(\chi_1,\chi_2,\tau)$ contains all extensions of $\chi_1$ by
$\chi_2$ that arise as the reduction mod $p$ of a potentially
Barsotti-Tate representation of $G_L$ of type $\tau$.  Note that if $L(\chi_1,\chi_2,\tau)
\neq \varnothing$ then  $(\chi_1 \chi_2)|_{I_{G_L}} = \lambda \lambda'
\epsilonbar$.

\begin{prop}
  \label{prop:minimal-exists}
Let $\mathcal{S}$ be the set of all $\M(r,a,c)$ of type
$\lambda\oplus \lambda'$ with generic fibre~$\chi$.  If $\mathcal{S}$ is
nonempty, then it has a minimal and a maximal element; that is, there
are Breuil modules $\M_{-},\M_{+} \in \mathcal{S}$ such that for any $\M \in
\mathcal{S}$ there exist nonzero maps $\M_{-} \to \M$ and $\M \to \M_{+}$.
\end{prop}

\begin{proof}
  By duality it suffices to prove the existence of $\M_{+}$.  For
  this, since
  $\cS$ is finite, it is enough to prove that any $\M,\N \in \cS$ have
  an upper bound in $\cS$, i.e.~that
  there exists $\cP \in \cS$ together with nonzero maps $\M \to \cP$
  and $\N \to \cP$.  

Since $\M,\N$ have the same generic fibre, the
  conditions of Proposition~\ref{prop:max-model} are satisfied, and we
  can form
  $\cP = \cP(t,a,v)$ as in the last sentence of the proposition.   Note that if
  $\gamma_i = \alpha_i$ then $v_i = c_i$, while if $\gamma_i =
  \beta_i$ then $v_i = d_i$ (see the last sentence of the proof of
  Proposition~\ref{prop:max-model}, for instance).   Thus
  $\sigma_i \circ \etabar^{v_i} \in \{  \sigma_i \circ \etabar^{c_i},  \sigma_i \circ \etabar^{d_i}
  \} \subseteq \{\lambda, \lambda'\}$, and we conclude that  $\cP \in \cS$.  
\end{proof}

\begin{remark}
  \label{rem:much-more-general}
  An argument identical to the above can be used to prove a much more general
  statement.  Namely, we can fix sets $S_i \subseteq \{ \sigma_i \circ
  \etabar^c \, : \, c \in \Z\}$ for each~$i$, and consider the set $\cS$ of Breuil modules
  $\M(r,a,c)$ with
  generic fibre $\chi$ such that $\sigma_i \circ \etabar^{c_i} \in
  S_i$ for all $i$; then if $\cS$ is nonempty, it has a maximal and
 a minimal element. 
\end{remark}

\begin{cor}
  \label{cor:is-vector-space}
  If $L(\chi_1,\chi_2,\tau)$ is nonempty, then it is a vector space.  
\end{cor}

\begin{proof}
  Suppose that $L(\chi_1,\chi_2,\tau)$ is nonempty.  By
  Proposition~\ref{prop:minimal-exists} there exists a minimal model
  $\M$ of type $\tau$ with generic fibre $\chi_1$.  It follows 
  easily 
   that $\M^{\vee}_{\chi}$ must be the maximal model of
  type $\tau$ with generic fibre $\chi_2$.
  Proposition~\ref{prop:containment} implies that $L(\chi_1,\chi_2,\tau)
  = L(\M,\M^{\vee}_{\chi})$, and the lemma follows.
\end{proof}

\subsection{The local setting.}\label{sec:mainsetting}

For remainder of the paper we suppose that $K/L$ is totally ramified
of degree $p^{f'} - 1$, so that $K = L(\pi)$, $f=f'$, and $e(K/L) = p^f-1$.
Recall that in this setting we have
 $\omega_i = (\sigma_i \circ \etabar) |_{I_L}$. The characters
 $\omega_i$ form a
 fundamental system of characters of niveau $f$, and we write
$$ \lambda = \prod_{i=0}^{f-1} \omega_i^{\nu_i},\qquad \lambda' =
\prod_{i=0}^{f-1} \omega_i^{\nu'_i}$$ with $\nu_i,\nu'_i \in [0,p-1]$
for all $i$; when either $\lambda$ or $\lambda'$ is trivial we require $\nu_i =
p-1$ for all $i$  or $\nu_i'=p-1$ for all $i$, respectively.  
Write $\lambda'/\lambda = \omega_0^{\delta}$ and define integers $\delta_i \in
[0,p-1]$ by
$\lambda'/\lambda = \prod_{i=0}^{f-1} \omega_i^{\delta_i}$,
with not all $\delta_i$ equal to $p-1$.   Let $[p^i \delta]$ be the
unique integer in the interval $[0,e(K/L)-1]$ congruent to $p^i \delta
\pmod{e(K/L)}$.

From the equality $\prod_{i=0}^{f-1} \omega_i^{\delta_i + \nu_i} =
\prod_{i=0}^{f-1} \omega_i^{\nu'_i}$ together with our bounds on the
$\delta_i,\nu_i,\nu'_i$, it follows that there exists a unique collection of integers $\gamma_i \in
\{0,1\}$ such that $\nu'_i = \delta_i + \nu_i - p \gamma_{i-1} + \gamma_{i}$.
We write $C = \{ i \, : \, \gamma_i = 1\}$ (the symbol $C$ here stands for ``carries'').

With the above notation,
we prove the following.

\begin{prop}
  \label{prop:generic-fibre-calculation}
   Suppose that $\M = \M(r,a,c)$ is a model of type $\lambda \oplus
  \lambda'$.  Let $J = \{ i : \sigma_i \circ \etabar^{c_i} \neq 
\lambda' \} \subseteq \{0,\ldots,f-1\}$.  Define $x_i$ by the formula
$$ r_i = \begin{cases}
 x_i e(K/L) & \text{if $i,i+1\in J$ or $i,i+1 \not\in J$}  \\
 x_i e(K/L) + [p^i \delta] & \text{if $i \in J$, $i+1 \not\in J$, $i
   \not\in C$} \\
 x_i e(K/L) + (e(K/L) - [p^i \delta]) & \text{if $i \not\in J$, $i+1 \in J$, $i
   \in C$} \\
 x_i e(K/L) - [p^i \delta] & \text{if $i \not\in J$, $i+1 \in J$, $i
   \not\in C$}\\
x_i e(K/L) - (e(K/L) - [p^i \delta]) & \text{if $i \in J$, $i+1 \not\in J$, $i
  \in C$}. 
\end{cases}$$
Then each $x_i$ is an integer in the interval $[0,e']$, and if $\lambda\neq
\lambda'$ then $x_i \neq e'$ in the second and third cases, while
$x_i\neq 0$ in the fourth and fifth cases. 
Moreover the generic fibre
of $\M$, on inertia, is equal to 
$$\prod_{i \in J} \omega_i^{\nu_i} \prod_{i \not\in J}
\omega_i^{\nu'_i} \prod_{i = 0}^{f-1} \omega_i^{x_i}.$$
\end{prop}

\begin{remark}
  \label{rem:smith}
  Note that $J = \{ i : \sigma_i \circ \etabar^{c_i} =
\lambda \}$ unless $\lambda = \lambda'$, in which case $J = \varnothing$.
 The special case of Proposition~\ref{prop:generic-fibre-calculation}
  where $\lambda = 1$ is given in
  \cite[\S2.2.1]{smiththesis}. The proof in
   \cite[\S2.2.1]{smiththesis} is essentially the same as the one we
  give here, but the statement of the result when $\lambda=1$ is somewhat simpler
  because $i \in C$ for all $i$.
\end{remark}

\begin{proof}[Proof of Proposition~\ref{prop:generic-fibre-calculation}]  The case $\lambda=\lambda'$ is straightforward (note
  that $\delta=0$, while $J = \varnothing$).  Assume for the rest of the proof
  that $\lambda\neq\lambda'$.  According to the definition of $J$ we
  have $p^{f-i} c_i \equiv \sum_{j=0}^{f-1}
  p^{f-j} \nu_j$ if $i \in J$ and  $p^{f-i} c_i \equiv \sum_{j=0}^{f-1}
  p^{f-j} \nu'_j$ if $i \not \in J$.
From the congruence $c_{i+1} \equiv
p(c_i + r_i) \pmod{e(K/L)}$ together with the definitions preceding
the statement of the Proposition, it follows that there exist integers
$y_i$ so that
$$r_i = \begin{cases}
 y_i e(K/L) & \text{if $i,i+1\in J$ or $i,i+1 \not\in J$}  \\
 y_i e(K/L) + [p^i \delta] & \text{if $i \in J$, $i+1 \not\in J$} \\
 y_i e(K/L) - [p^i \delta] & \text{if $i \not\in J$, $i+1 \in J$}.
\end{cases}$$
Since $r_i \in [0,e]$ for all $i$ we have in particular that $y_i \in [0,e']$, with $y_i \neq e'$ if
$i \in J, i+1 \not\in J$ and $y_i \neq 0$ if $i \not\in J, i+1\in J$.

From this formula for the $r_i$'s we calculate that
\begin{equation}\label{eq:first-step} \sum_{i=0}^{f-1} p^{f-i} r_i=
  \sum_{i=0}^{f-1} p^{f-i} y_i e(K/L) +
 \sum_{i \in J, i+1 \not\in J} p^{f-i} [p^i \delta] - \sum_{i \not\in
   J,i+1 \in J} p^{f-i} [p^i \delta].\end{equation}
Moreover we have $[p^i \delta] = \delta_i + p^{f-1} \delta_{i+1} + \cdots + p
\delta_{i-1}.$  Suppose that $0 \in J$.   
Let us compute the coefficient of $\delta_j$ on
the right-hand side of \eqref{eq:first-step}.  We see that $p^{f-i}[p^i \delta]$ contains a term of the form
$p^{f-j} \delta_j$ if $i \ge j$ and $p^{2f-j} \delta_j$ if $i < j$.
If $j \in J$ then the number of elements $i \in [0,j-1]$ such that $i \in J,
i+1\not\in J$ is equal to the number of elements $i \in [0,j-1]$ such that $i \not\in J,
i+1 \in J$, and similarly for the interval $[j,f-1]$.  It follows in
this case that
$\delta_j$ does not appear on the right-hand side of
\eqref{eq:first-step}.   If $j \not\in J$ then instead the
contribution of $\delta_j$ to the right-hand side of
\eqref{eq:first-step} is $(p^{2f-j} - p^{f-j})\delta_j$.  We conclude
that
$$ \alpha_0 = \frac{1}{p^f-1} \sum_{i=0}^{f-1} p^{f-i} r_i =
\sum_{i=0}^{f-1} p^{f-i} y_i + \sum_{i \not\in J} p^{f-i} \delta_i$$
and applying Lemma~\ref{lem:tst-evaluation} we find that the generic
fibre of $\M$, on inertia, is equal to 
$ \prod_{i=0}^{f-1} \omega_i^{\nu_i} \prod_{i \not\in J}
\omega_i^{\delta_i} \prod_{i=0}^{f-1} \omega_i^{y_i}$, which
rearranges 
to  \begin{equation}\label{eq:second-step}\prod_{i \in J}
\omega_i^{\nu_i} \prod_{i \not\in J} \omega_i^{\nu'_i} \prod_{i
  \not\in J} \omega_i^{p\gamma_{i-1} - \gamma_{i}} \prod_{i=0}^{f-1}
\omega_i^{y_i} \end{equation}
by substituting for
$\delta_i$ using the
defining formula for the $\gamma_i$'s.
An analogous calculation in the case $0 \not\in J$ yields the formula 
$$ \prod_{i \in J}
\omega_i^{\nu_i} \prod_{i \not\in J} \omega_i^{\nu'_i} \prod_{i
  \in J} \omega_i^{-p\gamma_{i-1} + \gamma_{i}} \prod_{i=0}^{f-1}
\omega_i^{y_i}$$
But $\prod_{i
  \not\in J} \omega_i^{p\delta_{i-1} - \delta_{i}} = \prod_{i
  \in J} \omega_i^{-p\delta_{i-1} + \delta_{i}}$ since
$\prod_{i=0}^{f-1} \omega_i^{p\delta_{i-1}-\delta_i} = 1$, so in fact the formula
\eqref{eq:second-step} is valid in all cases.  From the definition of
the set $C$ we can rewrite \eqref{eq:second-step}  as 
$$ \prod_{i \in J}
\omega_i^{\nu_i} \prod_{i \not\in J} \omega_i^{\nu'_i} \prod_{i \in C,
  i+1 \not\in J} \omega_i \prod_{i \in C,i \not\in J} \omega_i^{-1}
\prod_{i=0}^{f-1} \omega_i^{y_i}.$$
Now observe that with $x_i$ as in the statement of the Proposition we have $$x_i = \begin{cases} y_i+1 & \text{if $i \in J,i+1\not\in J,i\in
    C$} \\ y_i - 1 & \text{if $i\not\in J, i+1 \in J, i\in C$} \\
 y_i & \text{otherwise},\end{cases}$$
and the rest of the Proposition follows.
\end{proof}

\begin{defn}
  \label{defn:allowable}
   If $x_0,\ldots,x_{f-1}$ are integers in the interval $[0,e']$ with
   $x_i \neq e'$ whenever $i\in J,i+1 \not\in J,i \not\in C$ or
   $i\not\in J, i+1 \in J,i \in C$, and $x_i \neq 0$ whenever $i\in J,i+1 \not\in J,i \in C$ or
   $i\not\in J, i+1 \in J,i \not\in C$, we say that the $x_i$'s are
   \emph{allowable} for~$J$.  (Properly speaking we should say that
   they are allowable for $J$ and $C$, but~$C$ will remain fixed in
   any calculation.)  Observe  
    that for every choice of $J$ together  with a collection of
    $x_i$'s that are allowable for $J$,  there exists a model of type
   $\lambda \oplus \lambda'$ as in
   Proposition~\ref{prop:generic-fibre-calculation} that possesses those invariants.
\end{defn}

\begin{prop}
  \label{prop:minimal-trivial}
  Suppose that $\nu'_i \in [p-1-e',p-1]$ and $\nu_i \le \nu'_i$ for all
  $i$. 
  \begin{enumerate}
  \item There exists a model of type $\lambda \oplus \lambda'$
  with trivial generic fibre.

\item 
  The minimal model of type $\lambda
  \oplus \lambda'$ with trivial generic fibre is $\M = \M(r,1,c)$ with
  $r_i = e(K/L)(p-1-\nu'_i)$ and $c_i = \sum_{j=0}^{f-1} \nu'_{i-j}
  p^{j}$ for all $i$.   In the notation of
  Proposition~\ref{prop:generic-fibre-calculation} we have $J =
  \varnothing$ and $x_i = p-1-\nu'_i$ for all $i$.
  \end{enumerate}
\end{prop}

\begin{proof}  
If the generic fibre of $\M$ is trivial, by
Proposition~\ref{prop:generic-fibre-calculation} we can write
\begin{equation}\label{eq:identity} \prod_{i=0}^{f-1} \omega_i^{x_i} = \prod_{i\in J}
\omega_i^{p-1-\nu_i} \prod_{i \not\in J} \omega_i^{p-1-\nu'_i}\end{equation} and
conversely if this identity holds for some choice of $J$ and allowable $x_i$'s
then taking $\u{a} = 1$ gives a model of type $\lambda \oplus
  \lambda'$ with trivial generic fibre.  If we take $J =
 \varnothing$, then the integers $x_i = p-1-\nu_i' \in [0,e']$
  are automatically allowable; this proves (1), and since $J =
  \varnothing$ we have $\sigma_i \circ \etabar^{c_i} = \lambda'$ for
  all $i$, which implies $c_i = \sum_{j=0}^{f-1} \nu'_{i-j}
  p^{j}$.    It remains to show
  that the Breuil module $\M'$ corresponding to this data is actually
  the minimal model.  

First suppose that there exists a choice of $J$ and $x_i$'s so that
both sides of \eqref{eq:identity} are trivial.
Since $\nu_i,\nu'_i \in [0,p-1]$, this implies that at least one of
$\nu_i,\nu'_i$ is $p-1$ for all $i$, or at least one of $\nu_i,\nu'_i$
is $0$ for all $i$.   Since $\nu_i\le \nu'_i$, the latter would imply $\nu_i = 0$
for all $i$; but this contradicts our convention that $\nu_i = p-1$
for all $i$ when $\lambda = 1$.  So the former must hold, and we have 
$\nu'_i = p-1$ for all $i$.  Then $J =  \varnothing$
and $x_i = p-1-\nu'_i = 0$ for all $i$ evidently gives a minimal
model.

Now suppose it is never the case that both sides of \eqref{eq:identity} are
trivial.  Fix $J$ and integers $x_i \ge 0$ so that
\eqref{eq:identity} is satisfied (with $J = \varnothing$ if $\lambda=\lambda'$), define integers $r_i$ by the
formulas in the statement of
Proposition~\ref{prop:generic-fibre-calculation}, and then define
$\alpha_i = \frac{1}{p^f-1}\sum_{j=0}^{f-1} p^{f-j} r_{i+j} $ as
usual.  (Any model of type $\lambda \oplus \lambda'$ with
trivial generic fibre yields such data, with the $x_i$'s allowable for
$J$; however, note that in the
argument that follows we do \emph{not} assume that the
$x_i$'s are allowable for $J$.)  To deduce that $\M'$ is the minimal
model, it suffices by (the dual of) Proposition~\ref{prop:max-model}
to show that unless $J =
\varnothing$ and $x_i = p-1-\nu'_i$ for all $i$, we must have $\alpha_i > \alpha'_i$ for some $i$,
where the $\alpha'_i =\sum_{j=0}^{f-1} p^{f-j} (p-1-\nu'_{i+j})$ are the corresponding constants for $\M'$.

First suppose that $x_i \ge p$ for some $i$.  Replacing $x_i$ with
$x_i - p$ and $x_{i-1}$ with $x_{i-1}+1$ leaves the truth of
\eqref{eq:identity} unchanged, leaves $\alpha_j$ unchanged for all $j
\neq i$, and replaces $\alpha_i$ with $\alpha_i - p e(K/L)$.
By iterating this ``carrying'' operation we can reduce to the case
where $x_i \le p-1$ for all $i$.   In that case, since both sides
of \eqref{eq:identity} are assumed to be nontrivial  we must
actually have
$$
x_i = 
\begin{cases} 
p-1-\nu_i & \text{if $i \in J$} \\
p-1-\nu'_i & \text{if $i \not\in J$}.
 \end{cases}$$
The claim in the case $J = \varnothing$ is now immediate, so suppose that $J
\neq \varnothing$, and indeed suppose without loss of generality that
$0 \in J$.  
Note
that since $\nu'_i \ge \nu_i$ for all $i$ the set $C$ is empty, and
the proof of
Proposition~\ref{prop:generic-fibre-calculation} shows that
$\alpha_0$ is equal to
$\sum_{i=0}^{f-1} p^{f-i} x_i  + \sum_{i \not\in J} p^{f-i}
\delta_i$.    An inequality $\alpha_0 \le \alpha'_0$, or equivalently
$$\sum_{i=0}^{f-1} p^{f-i} x_i + \sum_{i \not\in J} p^{f-i} \delta_i
\le \sum_{i=0}^{f-1} p^{f-i} (p-1-\nu'_i),$$
would imply that 
$\nu_i = \nu'_i$ for all $i \in J$, and $\delta_i = 0$ for all $i
\not\in J$.  But $\delta_i = 0$ implies $\nu_i = \nu'_i$; so in fact
we would have 
$\nu_i = \nu'_i$ for all $i$, contradicting that $\lambda \neq
\lambda'$ when $J \neq \varnothing$.  Therefore $\alpha_0 > \alpha'_0$.
\end{proof}

\begin{cor}
  \label{cor:dim-count}
Let $\tau$ be a type as in Proposition~\ref{prop:minimal-trivial}.  Write $\nu'_i = (p-1-e') +
  \mu_i$ for all $i$.  If $\chi |_{I_{G_L}} = \lambda\lambda'\epsilonbar$
  then $$\dim_{k_E} L(1,\chi,\tau) \le \delta + \sum_{i=0}^{f-1}
  \mu_i$$ where $\delta = 1$ if $\chi = 1$ and $\delta  =0 $
  otherwise.
\end{cor}

\begin{proof}
  Let $\M$ be the minimal model of type $\tau$ with trivial generic
  fibre, as described by
  Proposition~\ref{prop:minimal-trivial}(2).  
  By the proof of 
  Corollary~\ref{cor:is-vector-space} we have $L(1,\chi,\tau) =
  L(\M,\M^{\vee}_{\chi})$.  We compute an upper bound on the dimension of $L(\M,\M^{\vee}_{\chi})$
  using Theorem~\ref{thm:extensions}.  Since $J = \varnothing$
  we have $r_i =
  (p-1-\nu'_i)e(K/L)$ for all $i$, and 
 $$s_i = e - r_i = e(K/L)e' - r_i = e(K/L) \mu_i.$$  Thus the $i$th
 term in the dimension formula in Theorem~\ref{thm:extensions} is $\mu_i$.
\end{proof}

We will now use Proposition~\ref{prop:comparison-with-minimax}
to compare the spaces $L(1,\chi,\tau)$ as $\tau$ varies, at least in certain cases.

\begin{prop} \label{prop:coagulation}  Let $\tau \simeq \lambda \oplus
  \lambda'$ be a type as in Proposition~\ref{prop:minimal-trivial},
  and suppose further that $\nu_i + \nu'_i \ge p-1$ for all $i$, and $\chi
  \neq 1$.
The space $L(1,\chi,\tau)$ is the set of extension classes
of generic fibres of Breuil modules of the form
 $ \cP(0,1,0;e,b,d^{\dagger}; h)$, where $d^{\dagger}_i =
 \sum_{j=0}^{f-1} (\nu_{i-j} +
 \nu'_{i-j} -(p-1)) p^j$, 
each $h_i$ is a polynomial whose only nonzero terms have degree
$t(p^f - 1) - [\sum_{j=0}^{f-1} (\nu_{i-j} + \nu'_{i-j} - (p-1))p^j]$
with $p-1-\nu'_i < t \le e'$, and $\Nm(\u{b})^{-1}$ gives the
unramified part of $\chi$ as in Lemma~\ref{lem:tst-evaluation}.
\end{prop}
\begin{proof}  Let $\M$ be the minimal model of
  Proposition~\ref{prop:minimal-trivial}.
Then $\M^{\vee}_{\chi} = \M(s,b,d)$ with $s_i = e(L/K) \mu_i$,
  $d_i = \sum_{j=0}^{f-1} \nu_{i-j} p^j$, and $\u{b}$ as in the
  statement of the proposition.
By Theorem~\ref{thm:extensions}, classes in $\Ext^1(\M,\N)$ have $h_i$ with terms of degree
$m(p^f-1) + \sum_{j=0}^{f-1} (\nu'_{i-j} - \nu_{i-j}) p^j$ with $0 \le
m < \mu_i$ (note that the hypotheses of
Proposition~\ref{prop:minimal-trivial} ensure that $\nu'_{i-j}
-\nu_{i-j}$ are non-negative and not all zero).

Now compute that the $\u{\delta}$ of 
Proposition~\ref{prop:comparison-with-minimax} has
$$\delta_i = ep/(p-1) - \beta_i + \alpha_i - r_i =
(p^f - 1)(p-1-\nu'_i) + 2\sum_{j=0}^{f-1} (p-1-\nu'_{i-j})p^j,$$
and so the terms of the Breuil module $\cP^{\dagger}$ of 
Proposition~\ref{prop:comparison-with-minimax}  have degree
$$m(p^f-1) +  \sum_{j=0}^{f-1} (\nu'_{i-j} - \nu_{i-j}) p^j + \delta_i
 = t(p^f-1) -  \sum_{j=0}^{f-1} (\nu_{i-j} +
 \nu'_{i-j} - (p-1))p^j$$
 where $t = p-1-\nu'_i +m+1$.  When $\nu'_i = \nu_i = p-1$ for all $i$
 (i.e.~in the unique case where  $[\sum_{j=0}^{f-1} (\nu_{i-j} +
 \nu'_{i-j} - (p-1))p^j]$ and $\sum_{j=0}^{f-1} (\nu_{i-j} +
 \nu'_{i-j} - (p-1))p^j$ are different) note that there is a change of
 basis parameter $\u{t}$ as in the proof of
 Theorem~\ref{thm:extensions}  with
$\u{t} \in (k \otimes k_E)^{\times}$ that exchanges the terms of degree $0$ in
 the $h_i$'s for
 terms of degree~$e'$.   One easily checks that $c^{\dagger}$ and $d^{\dagger}$
 are as claimed, completing the proof.
\end{proof}

\begin{remark}
  \label{rem:small-adjustment}
The Breuil modules $\cP$ of
Proposition~\ref{prop:coagulation} are usually in the canonical form
of Theorem~\ref{thm:extensions}; the exception is that if
$\lambda\lambda' = 1$ then we have terms of degree $e'$ in~$h_i$ instead of
terms of degree $0$.  However, as we have seen in the preceding argument, these are equivalent by a change of
basis parameter $\u{t}$ as in the proof of
Theorem~\ref{thm:extensions}.   
\end{remark}

\begin{cor} \label{cor:intersections}
For any $\tau,\dot{\tau}$ as in Proposition~\ref{prop:coagulation} and
$\chi\neq 1$ with $\chi |_{I_{G_L}} = \lambda\lambda'\epsilonbar$, we have
\begin{enumerate}
\item $\dim_{k_E} L(1,\chi,\tau) = \sum_{i=0}^{f-1} \mu_i$,
\item $L(1,\chi,\tau) \cap L(1,\chi,\dot{\tau}) = L(1,\chi,\ddot{\tau})$
where the type $\ddot{\tau}$ has $\ddot{\nu}_i =
\max(\nu_i,\dot{\nu}_i)$ and $\ddot{\nu}'_i =
\min(\nu'_i,\dot{\nu}'_i)$ (with the inferrable notation).
\end{enumerate}
\end{cor}
\begin{proof}
Let $\M_0 = \M(0,1,0)$ and $\N_0 = \M(e,b,d^{\dagger})$, with $b$ and
$d^{\dagger}$ as in the statement of Proposition~\ref{prop:coagulation}.
The map
$\Ext^1(\M_0,\N_0) \to \Ext^1_{k_E[G_L]}(1,\chi)$ is injective;
for instance, this follows from the dimension calculation in Theorem~\ref{thm:extensions}
together with the fact that the map is surjective except in the case of cyclotomic
$\chi$ when the image is the set  of peu ramifi\'ees classes.  Now
the result follows from Proposition~\ref{prop:coagulation} and
Corollary~\ref{cor:dim-count}, together with
Remark~\ref{rem:small-adjustment} in the case where  $\chi|_{I_{G_L}} =\epsilonbar|_{I_{G_L}}$. 
\end{proof}

\section{Weights and types}
\label{sec:weights-types}

\subsection{Serre weights}
\label{subsec:Serre weights}
We maintain the notation from the preceding section.  In particular, $L$ is a finite
extension of $\Q_p$ of absolute ramification degree $e'$, $K/L$ is totally ramified
of degree $p^{f'} - 1$, so that $K = L(\pi)$, $f=f'$, and $e(K/L) = p^f-1$.
We continue to assume
that the residue field is $k$ of degree $f$ over $\F_p$, and that $\sigma_i: k \into k_E$
are embeddings satisfying $\sigma_i = \sigma_{i+1}^p$ for $i = 0,\ldots,f-1$, taking
indices modulo $f$. 

Let $\rhobar: G_L \to \GL_2(k_E)$ be a reducible representation, so that
$$\rhobar \simeq  \left(\begin{array}{cc} \chi_2 & * \\ 0 & \chi_1 \end{array}\right)$$
for some characters $\chi_1,\chi_2 : G_L \to k_E^\times$.   In particular if
$\rhobar$ is decomposable, then we choose an ordering of the characters.
The ordered pair of characters $(\chi_1,\chi_2)$ will be fixed throughout
the section.

Recall that a {\em Serre weight} in our context is an isomorphism class
of absolutely irreducible representations of $\GL_2(k)$ in characteristic $p$. 
These are all defined over $k_E$ and have the form
$$\mu_{m,n} := \bigotimes_{i=0}^{f-1} 
  \left(\det{}^{m_i} \otimes \Sym^{n_i} k^2 \right) \otimes_{k,\sigma_i} k_E$$
where $m = (m_0,\ldots,m_{f-1})$ and $n = (n_0,\ldots,n_{f-1})$ are $f$-tuples
of integers satisfying $0 \le n_i \le p-1$ for all $i$.  The representations
$\mu_{m,n}$ and $\mu_{m',n'}$ are isomorphic if and only if $n = n'$
and $\sum_{i=0}^{f-1} m_{i}p^{f-i} \equiv \sum_{i=0}^{f-1} m'_{i}p^{f-i}
\pmod{p^f - 1}$.
  
A set of predicted Serre weights for $\rhobar$ is defined by
Barnet-Lamb, Gee and Geraghty in~\cite[Def.~4.1.14]{blggU2} (building
on~\cite{bdj,ScheinRamified,GeePrescribed}, and following \cite{EGHS}).  In order to give
the definition, we use the notion of a Hodge--Tate module.

\begin{defn}  \label{defn:HTtype}
A {\em Hodge--Tate module of rank $d$ (for $L$ over $E$)}
is an isomorphism class of filtered free $(L\otimes_{\Q_p}E)$-modules
of rank $d$, i.e.~of objects $(V,\Fil^\bullet)$, where 
$V$ is a free $(L\otimes_{\Q_p}E)$-module of rank $d$ and for $i \in \Z$, 
$\Fil^iV$ is a (not necessarily free) $(L\otimes_{\Q_p}E)$-submodule
such that $\Fil^j V \subseteq \Fil^iV$ if $i\le j$, $\Fil^i V = V$ for $i \ll 0$
and $\Fil^iV = 0$ for $i \gg 0$.
\end{defn}
Recall that we are assuming that $E$ contains all the embeddings
of $L$ in $\Qpbar$, so to give a Hodge--Tate module of rank $d$
is equivalent to giving, for each $\sigma: L \into \Qpbar$,
a $d$-tuple of integers $(w_{\sigma,1},\ldots,w_{\sigma,d})$
with $w_{\sigma,1} \le w_{\sigma,2} \le \cdots \le w_{\sigma,d}$.
For consistency with our conventions we normalize this
correspondence so that $(V,\Fil^\bullet)$ corresponds to the
$d$-tuples $(w_{\sigma,1},\ldots,w_{\sigma,d})$ defined by
$$-w_{\sigma,r} = \max\{\,w\,:\,  r \le \dim_E
\Fil^w(V\otimes_{L\otimes_{\Q_p} E } E) \,\},$$
where the tensor product is relative to the projection $L\otimes_{\Q_p} E \to E$
defined by $x \otimes y \mapsto \sigma(x)y$. 

\begin{defn}
  \label{defn:HTwts}
 We refer to the $d$-tuple
$(w_{\sigma,1}, w_{\sigma,2}, \ldots, w_{\sigma,d})$ as the
{\em $\sigma$-labelled Hodge--Tate weights} 
of $(V,\Fil^\bullet)$.
We say that $(V,\Fil^\bullet)$ is a {\em lift} of the Serre weight
$\mu_{m,n}$ if $d=2$ and for each $i=0,\ldots,f-1$ there is an embedding
$\tilde{\sigma}_i:L \into E$ lifting $\sigma_i$ such that:
\begin{itemize}
\item $(V,\Fil^\bullet)$ has $\tilde{\sigma}_i$-labelled Hodge--Tate weights $(m_i,m_i+n_i+1)$;
\item for each $\sigma \neq \tilde{\sigma}_i$ lifting $\sigma_i$,
$(V,\Fil^\bullet)$ has $\sigma$-labelled Hodge--Tate weights $(0,1)$.
\end{itemize}
Recall that if $\rho: G_L \to \GL_d(E)$ is crystalline, then $D_\dR(\rho)$ has the
structure of a filtered free $(L\otimes_{\Q_p}E)$-module of rank $d$ as in
Definition~\ref{defn:HTtype}.  We then define the Hodge--Tate module and
$\sigma$-labelled Hodge--Tate weights of $\rho$ to be those 
of $D_\dR(\rho)$.
  \end{defn}

\begin{defn}  \label{defn:predicted_weight}
We say that $\mu$ is a {\em predicted Serre weight} for $\rhobar$ if,
enlarging $E$ if necessary, 
$\rhobar$ has a reducible crystalline lift $\rho$
whose Hodge--Tate type is a lift of $\mu$.  We then define $W_{\expl}(\rhobar)$ to
be the set of predicted Serre weights for $\rhobar$.
\end{defn}

It is immediate from the definition that $W_{\expl}(\rhobar)\subset
W_{\expl}(\rhobar^\ssimp)$; moreover it follows from the description
of reductions of crystalline characters that
$W_{\expl}(\rhobar^\ssimp)$ is precisely the set of Serre weights
for $\rhobar^\ssimp$ predicted by Schein in \cite{ScheinRamified}
 (see \cite[Lemma~4.1.22]{blggU2}).
Recall that this is the set of $\mu_{m,n}$ such that 
\begin{equation}\label{eq:Schein_weights}
\begin{array}{rcl}
\chi_2|_{I_{G_L}} &=&
\prod_{i\in J}\omega_i^{m_i + n_i+e'-d_i}\prod_{i\not\in J}\omega_i^{m_i + e' - d_i}\\
\quad \text{and} \quad \chi_1|_{I_{G_L}} &= &
\prod_{i\in J}\omega_i^{m_i+d_i}\prod_{i\not\in J}\omega_i^{m_i + n_i + d_i}
\end{array}
\end{equation}
for some $J \subseteq \{0,\ldots,f-1\}$ and integers $d_i$ for $i =0,\ldots,f-1$
satisfying $0 \le d_i \le e'-1$ if $i\in J$ and $1 \le d_i \le e'$ if $i\not\in J$.
Thus $W_{\expl}(\rhobar^\ssimp)$ is indeed ``explicit,'' as the notation
is presumably meant to indicate; however $W_{\expl}(\rhobar)$ is less
so since it is defined in terms of reductions of extensions of crystalline
characters.

\subsection{A partition by types}
\label{subsec:partition}
We fix $\barrho$ as in \S\ref{subsec:Serre weights} and let
$W' = W_{\expl}(\rhobar^\ssimp)$.
The aim of this section is to define a partition of $W'$
under the following hypothesis on $\rhobar$:
\begin{defn}
  \label{defn:generic}
We say that $\rhobar$ is
{\em generic} if $\chi_1^{-1}\chi_2|_{I_{G_L}} = \prod_{i=1}^f \omega_i^{b_i+e'}$
for some integers $b_i$ satisfying
$$e' \le b_i + e' \le p- 1 - e'.$$
\end{defn}

We assume for the remainder of the paper that $\rhobar$ is generic, so that we have integers
$b_i$ as above.\footnote{It appears to us that it should be possible
  to replace this genericity hypothesis with a somewhat weaker
  hypothesis and still prove the main results of this
paper (using the methods of this paper). Indeed no such hypothesis
was needed in the totally ramified case \cite{GLS11}. On the other hand the discussion in
Section~\ref{sec:counterexample} below shows that with these methods  
one cannot expect to remove the genericity hypothesis entirely even for $L = \Q_{p^2}$, 
and so we have to some extent favored cleaner combinatorics over optimizing the
genericity hypothesis.}  Note in particular that this implies that $e' \le
(p-1)/2$.  We also write $\chi_1|_{I_{G_L}}= \prod_{i=1}^f \omega_i^{c_i}$
for some integers $c_i$.

Suppose that $\mu_{m,n} \in W'$, with $J$ and $d = (d_0,\ldots,d_{f-1})$ as 
in (\ref{eq:Schein_weights}).
Then $n$ satisfies the congruence:
$$\sum_{i=0}^{f-1} (b_{i} + 2d_{i})p^{f-i} \equiv
 \sum_{i\in J} n_{i}p^{f-i} - \sum_{i\not\in J} n_{i}p^{f-i} \pmod{p^f  -1}.$$
 One easily sees that given $J$ and $d$, there is a unique such $n$ unless
 $$\sum_{i=0}^{f-1} (b_{i} + 2d_{i})p^{f-i} \equiv \sum_{i \in J} (p-1)p^{f-i}\pmod{p^f  -1}.$$
 The genericity hypothesis implies that $0 \le b_i + 2d_i < p - 1$ if $i\in J$,
 and $0 < b_i + 2d_i \le p - 1$ if $i \not\in J$, so we see that $n$ is unique
 unless either $b = d = (0,\ldots,0)$ and $J=\{0,\ldots,f-1\}$, or 
 $b = (p-1-2e',\ldots,p-1-2e')$, $d = (e',\ldots,e')$ and $J =
 \varnothing$ (and so in particular unless $\chi_1^{-1} \chi_2
|_{I_{G_L}} = \epsilonbar|_{I_{G_L}}^{\pm 1}$).
 It follows that aside from these two exceptional cases, there is a
 unique $\mu_{m,n}$ for each pair $(J,d)$, and one checks that it is given by:
 \begin{equation}\label{eq:explicit-weights}\begin{array}{lll}
 m_i = c_i + p -  1 - d_i,& n_i = b_i + 2d_i,&\mbox{if $i\in J$ and $i+1\in J$;}\\
 m_i = c_i + p -  1 - d_i,& n_i = b_i + 2d_i + 1,&\mbox{if $i\in J$ and $i+1\not\in J$;}\\
 m_i = c_i + b_i + d_i - 1,& n_i = p - b_i - 2d_i,&\mbox{if $i\not\in J$ and $i+1\in J$;}\\
 m_i = c_i + b_i + d_i ,& n_i = p - 1 - b_i - 2d_i,&\mbox{if $i\not\in J$ and $i+1\not\in J$.}
  \end{array}\end{equation}
We let $\mu(J,d)$ denote the weight $\mu_{m,n}$ with $m,n$ defined by
(\ref{eq:explicit-weights}).   In the two exceptional cases, we obtain in addition
to $\mu(J,d)$ the weight $\mu'(J,d)$ defined as follows:
if $b =d = (0,\ldots,0)$ and $J=\{0,\ldots,f-1\}$, then $\mu'(J,d)=
\mu_{m,n}$ where $m_i = c_i$ and $n_i = p -1$ for all $i$,
and if $b=(p-1-2e',\ldots,p-1-2e')$, $d=(e',\ldots,e')$ and $J=\varnothing$,
then $\mu'(J,d)= \mu_{m,n}$ where
$m_i = c_i - e'$ and $n_i = p - 1$ for all $i$.

We let $W$ denote the subset of $W'$ consisting of the $\mu(J,d)$.
Note also that for $(m,n)$ as in (\ref{eq:explicit-weights}), we always
have $n_i < p-1$ for all $i$.  It follows that the additional weights
$\mu'(J,d)$ (when they occur) are not in $W$.  Note also that
both additional weights arise if $b=(0,\ldots,0)$ and $e' = (p-1)/2$, but
comparing values of $m$ shows they are distinct from each other.
Moreover the following lemma shows that the weights $\mu(J,d)$ are distinct.
\begin{lemma} \label{lem:distinct-weights} Suppose $J,J' \subseteq S$  and
that $d = (d_0,\ldots,d_{f-1})$
and $d' = (d_0',\ldots,d_{f-1}')$ are $f$-tuples of integers satisfying
$0 \le d_i \le e' - 1$ if $i \in J$, $1 \le d_i \le e'$ if $i\not\in J$, 
$0 \le d'_i \le e' -1$ if $i\in J'$ and $1 \le d'_i \le e'$ if $i\not\in J'$.
If $\mu(J,d)$ is isomorphic to $\mu(J',d')$, then $J = J'$ and $d=d'$.
\end{lemma}
\begin{proof}    Write $\mu(J,d) = \mu_{m,n}$ and
$\mu(J',d') = \mu_{m',n'}$ with $(m,n)$ and $(m',n')$ as in
(\ref{eq:explicit-weights}).   Twisting by $\chi_1^{-1}$,
we may suppose that $c_i = 0$ for all $i$.   Then $0 \le m_i \le p - 1$
for all $i$, and $m_i > 0$ for some $i$, so that
$0 < \sum_{i=1}^{f} m_{i}p^{f-i} \le p^f - 1$.  Since the same is true
for $m'$, we must have $m_i = m_i'$ for all $i$.  We claim that $J = J'$.
Indeed if not, then without loss of generality there is some $i \in J$
such that $i \not\in J'$, but then
$$m_i = p - 1 - d_i \ge p - e' > b_i + d_i \ge m_i',$$
giving a contradiction.  Since $J = J'$ and $m = m'$, it follows immediately
that $d = d'$.
\end{proof}

We will now define partitions of $W$ and $W'$ into
subsets indexed by $A$, where 
$A$ is the set of $f$-tuples $a = (a_0,a_1,\ldots,a_{f-1})$
with $0 \le a_i \le e'$ for all $i$.    For $a \in A$, we let $\tau_a$ denote the (at most) tamely ramified principal series inertial type 
$$\tau_a :=   \prod_{i=0}^{f-1} \om_i^{c_i+b_i+a_i}  \oplus  \prod_{i=0}^{f-1} \om_i^{c_i-a_i},$$
where $\om_i$ denotes the Teichm\"uller lift of $\omega_i$.

If $\tau$ is a principal
series type, we let $\theta_\tau$ denote the $\GL_2(\cO_L)$-type
associated to
$\tau$ by the inertial local Langlands correspondence, viewed as a
representation of $\GL_2(k)$.  If $\tau =
\tau_a$ then we write $\theta_a$ for $\theta_\tau$;
 so if $\tau_a$ is non-scalar then explicitly
$$\theta_a = \Ind_B^{\GL_2(k)} \left( \prod_{i=0}^{f-1} \om_i^{c_i+b_i+a_i}
\otimes  \prod_{i=0}^{f-1}  \om_i^{c_i-a_i}\right)$$
where $B$ is the subgroup of upper-triangular matrices in $\GL_2(k)$,
$\psi_1\otimes\psi_2$ denotes the character of $B$ sending
$\left(\begin{array}{cc} x & * \\ 0 & y \end{array}\right)$ to
$\psi_1(x)\psi_2(y)$, and we recall that we are identifying characters
$k^{\times} = \ell^{\times} \to \Qpbar^{\times}$ with characters $I_L
\to \Qpbar^{\times}$ via the local Artin map with its geometric normalisation.
Note that if $\tau_a$ is scalar, then 
$\sum_{i=0}^{f-1} (b_{i} + 2a_{i})p^{f-i} \equiv 0 \pmod{p^f - 1}$, which occurs
only if $a = b = (0,\ldots,0)$, or $a=(e',\ldots,e')$ and $b =
(p-1-2e',\ldots,p-1-2e')$.
In this case we let
$$\theta_{\tau_a} = \theta_a = \det\circ \prod_{i=0}^{f-1} \om_i^{c_i+b_i+a_i},\quad\mbox{and}
\quad \theta'_{\tau_a} = \theta'_a = \theta_a\otimes \Ind_B^{\GL_2(k)}  \mathbf{1}.
$$
We then define
$$\begin{array}{rcl}
W_a &:=& \{\, \mu \in W' \,:\, \mbox{$\mu$ is a Jordan--H\"older constituent of $\thetabar_a$}\,\},\\
\mbox{and}\ W'_a&:=&
 \{\, \mu \in W' \,:\, \mbox{$\mu$ is a Jordan--H\"older constituent of $\thetabar'_a$}\,\}.\end{array}.$$
We will see shortly that $W_a$ is in fact contained in $W$.  Note that $W'_a = W_a$ unless $a=b=(0,\ldots,0)$ in which case
$W_a = \{\mu(J,d)\}$ and $W'_a = \{\mu(J,d),\mu'(J,d)\}$ with $J= \{0,\ldots,f-1\}$ and
$d=(0,\ldots,0)$, or $a=(e',\ldots,e')$ and $b=(p-1-2e',\ldots,p-1-2e')$ in which case
$W_a = \{\mu(J,d)\}$ and $W'_a = \{\mu(J,d),\mu'(J,d)\}$ with $J= \varnothing$ and
$d=(e',\ldots,e')$.

\begin{prop}\label{prop:partition} \
\begin{enumerate}
\item $W$ (resp. $W'$) is the disjoint union of the $W_a$ (resp. $W'_a$) for $a\in A$.
\item $|W_a| = 2^{f-\delta_a}$ where 
$\delta_a = \left|\{\,i\in\{0,\ldots,f-1\}\, : \,\mbox{$a_i=0$ or $e'$}\,\}\right|$.
\item If $\mu(J,d)$ or $\mu'(J,d) \in W'_a$, then $\sum_{i=0}^{f-1}a_i = \sum_{i=0}^{f-1} d_i$.
\end{enumerate}
\end{prop}
\begin{proof} First note that to prove the proposition, we may twist $\barrho$ so as
to assume $c_i = 0$ for $i=0,\ldots,f-1$.

To prove (1), we must show that for each $(J,d)$ as in the definition
of $W$, there is a unique $a \in A$ such that $\mu(J,d)$ is a Jordan--H\"older constituent of
$\thetabar_a$.  For this we use the explicit description of $\thetabar_a^\ssimp$ given for
example in \cite[Prop.~1.1]{DiamondDurham}.  In particular the Jordan--H\"older constituents are
of the form $\nu(a,J')$ for certain subsets $J' \subseteq \{0,\ldots,f-1\}$, where 
$\nu(a,J') = \mu_{m',n'}$ with $m' = (m_0',\ldots,m_{f-1}')$ and $n' = (n_0',\ldots,n_{f-1}')$ defined by
\begin{equation}\label{eq:explicit-JH}\begin{array}{lll}
 m'_i = p -  1 - a_i,& n'_i = b_i + 2a_i,&\mbox{if $i\in J'$ and $i+1\in J'$;}\\
 m'_i =  p -  a_i,& n'_i = b_i + 2a_i - 1,&\mbox{if $i\in J'$ and $i+1\not\in J'$;}\\
 m'_i = b_i + a_i,& n'_i = p - 2 - b_i - 2a_i,&\mbox{if $i\not\in J'$ and $i+1\in J'$;}\\
 m'_i = b_i + a_i ,& n'_i = p - 1 - b_i - 2a_i,&\mbox{if $i\not\in J'$ and $i+1\not\in J'$.}
  \end{array}\end{equation}
The constituents are then precisely the $\nu(a,J')$ for those $J'$ such that $n'_i \ge 0$ for all $i$,
except in the case that $\tau_a$ is scalar, in which case there is only one
Jordan--H\"older constituent, namely $\nu(a,J')$ with $J' = \{0,\ldots,f-1\}$ 
(resp.~$J' = \varnothing$) if  $a = b = (0,\ldots,0)$  
(resp.~$a=(e',\ldots,e')$ and $b = (p-1-2e',\ldots,p-1-2e')$).

Suppose now that $\nu(a,J') = \mu_{m',n'} \simeq \mu_{m,n} = \mu(J,d)$.
Note that $0 \le m_i' \le p$ for $i=0,1,\ldots,f-1$; we will rule out the
possibility that $m_i' = p$ for some $i$.  Indeed if  $m'_{i} = p$, then
we must have $a_{i} = 0$, $i \in J'$ and  $i+1 \not\in   J'$. It follows
that $m'_{i+1} = b_{i+1} + a_{i+1} \le p -2$, so that
$$ 0  <  \sum_{j=0}^{f-1} m'_{i-j}p^j  < p^f - 1.$$
Since $0 \le m_i \le p - 1$ for all $i$ (and not all $0$), and
$$ \sum_{j=0}^{f-1} m_{i-j}p^j \equiv  \sum_{j=0}^{f-1} m'_{i-j}p^j  \pmod{p^f - 1},$$
we see that the sums are equal.  Therefore $m_{i} \equiv m'_{i} \equiv 0 \pmod{p}$,
so in fact $m_{i} = 0$.  The definition of $m_{i}$ then implies that $b_{i} = 0$,
giving $n'_{i} = -1$, a contradiction.  Note also that if $m'_i = 0$ for all $i$,
then $a = b = (0,\ldots,0)$ and $J' = \varnothing$, which is also impossible.
Since
$$\sum_{i=1}^{f} m_{i}p^{f-i} \equiv  \sum_{i=1}^{f} m'_{i}p^{f-i}  \pmod{p^f - 1}$$
and both sums are between $1$ and $p^f-1$ (inclusive), it follows that
$(m,n) = (m',n')$.

Next we show that $J' = J$.   If $i \in J$ and $i \not\in J'$ for some $i$, then
$$m_i' = a_i + b_i \le p - 1 - e' < p - 1 - d_i = m_i,$$
giving a contradiction.  If $i \not\in J$ and $i \in J'$ for some $i$, then
the inequalities
$$m_i \le b_i + d_i \le p - 1 - e' \le p - 1 - a_i \le m_i'$$
must all be equalities, which implies that
$i + 1 \not\in J$, $i + 1 \in J'$,  $b_i = p - 1 - 2e'$ and $a_i = e'$.
Iterating gives $J' = \{0,\ldots,f-1\}$, $b = (p-1-2e',\ldots, p-1-2e')$ and $a = (e',\ldots,e')$,
which is impossible.  

Having shown that $J' = J$, it follows that $a$ is determined by the equation
\begin{equation} \label{eqn:a-vs-d} a_i = \left\{\begin{array}{ll}
d_i+1,&\mbox{if $i\in J$, $i+1\not\in J$}, \\
d_i-1,&\mbox{if $i\not\in J$, $i+1\in J$}, \\
d_i,&\mbox{otherwise.}\end{array}\right.\end{equation}
As indeed $(m',n') = (m,n)$ in this case (as well as $n_i \neq -1$ and
$a_i \in [0,e']$ for all $i$), this gives the assertion
for $W$.  The assertion for $W'$ follows upon checking that when $b = (0,\ldots,0)$ (resp. $b =
(p-1-2e',\ldots,p-1-2e')$) the constituent $\mu'(J,d)$ with $J =
\{0,\ldots,f-1\}$ and $d = (0,\ldots,0)$ (resp. $J = \varnothing$ and
$d = (e',\ldots,e')$) is not contained in $W_a$ with $a \neq
(0,\ldots,0)$ (resp. $a \neq (e',\ldots,e')$).

To prove (2) we fix $a$ and determine the $J \subseteq \{0,\ldots, f- 1\}$ for 
which (\ref{eqn:a-vs-d}) holds for some $d$ as in the definition of $W$.
The condition that $0 \le d_i \le e' - 1$ if $i \in J$ and $1 \le d_i \le e'$
if $i\not\in J$ translates into the condition that $0 \le a_i \le e' - 1$ if $i+1 \in J$,
and $1\le a_i \le e'$ if $i+1 \not\in J$.  Therefore the only restrictions on $J$ are
that  $i+1 \in J$ if $a_i = 0$, and that $i + 1 \not\in J$ if $a_i = e'$. The number of
such $J$ is $2^{f-\delta_a}$ as required.

Part (3) in the case $\mu(J,d) \in W'_a$ is immediate from (\ref{eqn:a-vs-d})
on noting that there are the same number of $i$ satisfying $i\in J$, $i+1\not\in J$
as there are satisfying $i\not\in J$, $i+1\in J$.   The formula in the case
$\mu'(J,d) \in W'_a$ is immediate from the definitions.
\end{proof}

\begin{rem}
  \label{rem:schein-partition}
  We remark that Schein~\cite[Prop.~3.2]{MR2482000} also gives a
  decomposition of $W'$ into subsets which are typically of cardinality $2^f$,
  but it is visibly different from ours; for example, 
   it is a decomposition into $(e')^f$ subsets rather than
  $(e'+1)^f$, and if $f=1$, then the subsets are constituents of the reduction of a
  supercuspidal rather than principal series type.
\end{rem}

Recall that $L(\chi_1,\chi_2,\tau)$ denotes the set of all extensions
of $\chi_1$ by $\chi_2$ that arise as the generic fibre of a model of
type $\tau$.  We translate Corollary~\ref{cor:intersections} into the
notation of this section.

\begin{thm} \label{thm:intersections-redux}
For any $a,a' \in A$, we have
\begin{enumerate}
\item $\dim_{k_E} L(\chi_1,\chi_2,\tau_a) = \sum_{i=0}^{f-1} (e' - a_i)$,
\item $L(\chi_1,\chi_2,\tau_a) \cap L(\chi_1,\chi_2,\tau_{a'}) = L(\chi_1,\chi_2,\tau_{a''})$
where $a_i'' = \max(a_i,a_i')$.
\end{enumerate}
\end{thm}
\begin{proof}
Reduce to the case of $\chi_1 = 1$ by twisting.  Our genericity
hypothesis rules out $\chi_2=1$.  Note that for the
type $\tau_a$ we have $\nu'_i = p-1-a_i$ and $\nu_i = a_i + b_i$,
except that when $a = b = 0$ we (by convention) have $\nu_i = p-1$ for all $i$.  The
conditions $\nu'_i \in [p-1-e',p-1]$, $\nu'_i \ge \nu_i$, and $\nu'_i
+ \nu_i \ge p-1$ are all easily checked.  Now $\mu_i = e' - a_i$ and
(in the notation of Corollary~\ref{cor:intersections})
if $(\tau,\dot{\tau}) = (\tau_a,\tau_{a'})$ then $\ddot{\tau} =
\tau_{a''}$, as desired.
\end{proof}

\section{The main results}\label{sec:main-results}
\subsection{The global setting}
\label{subsec:global setting}
Let $F$ be a totally real field and $\rhobar:G_F \to \GL_2(k_E)$ a continuous
representation.  We suppose that $\rhobar$ is automorphic in the
sense that it arises as the reduction of a $p$-adic representation of $G_F$
associated to a cuspidal Hilbert modular eigenform of some weight and level,
or equivalently to a cuspidal holomorphic automorphic representation of
$\GL_2(\A_F)$.  We fix a place $v$ of $F$ dividing $p$, and we let $L
= F_v$, so that $k_v = k = \ell$ in what follows.

Let $D$ be a quaternion algebra over $F$ satisfying the following hypotheses:
\begin{itemize}
\item $D$ is split at all primes dividing $p$;
\item $D$ is split at at most one infinite place of $F$;
\item If $w$ is a finite place of $F$ at which $D$ is ramified, then
$\rhobar|_{G_{F_w}}$ is either irreducible, or equivalent to a representation of the form
$\psi \otimes \left(\begin{array}{cc} \epsilonbar & * \\ 0 & 1 \end{array}\right)$
for some character $\psi:G_F \to k_E^\times$.
\end{itemize}
 
 We let $r$ denote the number of infinite places of $F$ at which $D$ is split  (so $r = 0$ or $1$),
 and if $r = 1$ we let $\xi$ denote that infinite place and fix an isomorphism 
 $D_\xi \simeq M_2(\R)$.   We also fix a maximal order $\cO_D$ of $D$ and an
 isomorphism $\cO_{D_v} \simeq M_2(\cO_L)$. 
 
 \begin{rem} \label{rem:splitD_p}
 The hypothesis that $D$ is split at all primes dividing $p$ is made
 only to be able to invoke the results of \cite{GeeKisin} on the weight part of
 Serre's Conjecture.  
   We expect however that the proofs of the required variants of their results,
  and hence the main results of this paper, carry over if we only require
 that $D$ is split at $v$,
 without specifying the behavior of $D$ at the other primes dividing $p$.
 \end{rem}
   
 For any open compact subgroup $U$ of $D_f^\times = (D\otimes\widehat{\Z})^\times$, 
 we let $X_U$ denote the associated Shimura variety of dimension $r$:
 $$X_U = D^\times \backslash ( (\frakH^\pm)^r \times D_f^\times)/ U,$$
 where if $r =1$ then $D^\times$ acts on $\frakH^\pm = \C - \R$ via the isomorphism $D_\xi^\times \simeq
 \GL_2(\R)$, and we let $S^D(U) = H^r(X_U,k_E)$.   
 (Recall that $r$ and $\xi$ are defined just before Remark~\ref{rem:splitD_p}.)
 Let $\Sigma_U$ denote
 the set of all finite places $w$ of $F$ such that (i) $w$ does not divide $p$,
 (ii) $D$ is split at $w$, (iii) $U$ contains $\cO_{D_w}^\times$, and
 (iv) $\rhobar$ is unramified at $w$.  Then
 $S^D(U)$ is equipped with the commuting action of Hecke operators
 $T_w$ and $S_w$ for all $w \in \Sigma_U$, hence with an action
 of the polynomial algebra over $k_E$ generated by variables
 $T_w$ and $S_w$ for $w \in \Sigma_U$.  We denote this algebra
 by $\TT^{\Sigma_U}$, and let $\m_{\rhobar}^{\Sigma_U}$ denote the kernel of the
 $k_E$-algebra homomorphism $\TT^{\Sigma_U} \to k_E$ defined by
 $$T_w \mapsto  \Nm(w)\Trace(\rhobar(\Frob_w));
  \qquad S_w \mapsto  \Nm(w)\det(\rhobar(\Frob_w))$$
  for $w \in \Sigma^U$.
  We let $S^D(U)[\m_{\rhobar}^{\Sigma_U}]$ denote that set of $x \in S^D(U)$
  such that $Tx = 0$ for all $T\in \m_{\rhobar}^{\Sigma_U}$.
 
 Now let $U_v$ denote the kernel of the map $\cO_{D_v}^\times \to
 \GL_2(\kv)$ 
 defined
 by composing the restriction of our fixed $\cO_{D_v} \simeq M_2(\cO_L)$ with reduction
 mod $v$. If $U = U_vU^v$ for some open compact $U^v \subseteq \ker(D_f^\times \to D_v^\times)$,
 then the natural right action of $\cO_{D,v}^\times$ on $X_U$ induces
 a left action of $\GL_2(\kv)$ on $S^D(U)$ which commutes with that of 
 $\TT^{\Sigma_U}$.
 
 \begin{defn}  \label{defn:modular_weights}
 If $\mu$ is an irreducible representation of $\GL_2(\kv)$ over $k_E$,
 then we say that $\rhobar$ is {\em modular of weight $\mu$ with
   respect to $D$ and $v$} 
 if
 $$\Hom_{k_E[\GL_2(\kv)]} (\mu, S^D(U)[\m_{\rhobar}^{\Sigma_U}]) \neq 0$$
 for some open compact subgroup $U = U_vU^v$ as above.
 We let $W^{D,v}_{\modular}(\rhobar)$ denote the set of Serre
weights for which $\rhobar$ is modular with respect to $D$ and $v$.
\end{defn}

The weight part of Serre's Conjecture for $\rhobar$ (at $v$, with respect to $D$)
states that
\begin{equation} \label{eqn:weight_conj}
 W^{D,v}_{\modular}(\rhobar) = W_{\expl}(\rhobar|_{G_L}),
 \end{equation}
where $W_{\expl}(\rhobar|_{G_L})$
is the set of predicted Serre weights as in~\cite[Def.~4.1.14]{blggU2},
recalled in Definition~\ref{defn:predicted_weight} above
in the case that $\rhobar|_{G_L}$ is reducible.

One of the inclusions in (\ref{eqn:weight_conj}) is proved under mild
technical hypotheses by Gee and Kisin in \cite{GeeKisin}, building
on \cite{geeBDJ, geesavitttotallyramified, GLS11, blggU2, GLS2}.
More precisely, we have the following result
(\textit{cf.}~\cite[Def.~5.5.3, Cor.~5.5.4]{GeeKisin}):
\begin{thm} \label{thm:GeeKisin}
Suppose that $p>2$, $\rhobar|_{G_{F(\zeta_p)}}$ is irreducible,
and if $p=5$, then the projective image of $\rhobar|_{G_{F(\zeta_5)}}$ is not
isomorphic to $A_5$.  Then the following hold:
\begin{enumerate}
\item $W^{D,v}_{\modular}(\rhobar)$ depends only on $\rhobar|_{G_L}$;
\item $W_{\expl}(\rhobar|_{G_L}) \subseteq W^{D,v}_{\modular}(\rhobar)$;
\item $W_{\expl}(\rhobar|_{G_L})  = W^{D,v}_{\modular}(\rhobar)$ if $L$ is
unramified  or totally ramified over $\Q_p$.
\end{enumerate}
\end{thm}
In particular the theorem ensures (under its hypotheses)
that $W^{D,v}_{\modular}(\rhobar)$ is independent of the
choice of $D$, which we henceforth suppress from the notation.

We will now restrict to the case where $\rhobar|_{G_L}$ is reducible and generic
(see Definition~\ref{defn:generic}).   Our main global result is the following:
\begin{thm}\label{thm:mainglobalresult}  If $\rhobar$ is as in Theorem~\ref{thm:GeeKisin} and $\rhobar|_{G_L}$
is reducible and generic, then
$$W_\expl(\rhobar|_{G_L}) = 
     W^v_\modular(\rhobar) \cap W_\expl(\rhobar|_{G_L}^\ssimp).$$
\end{thm}
In other words we prove the weight part of Serre's Conjecture holds
in this case for weights in $W_\expl(\rhobar|_{G_L}^\ssimp)$.
We will prove this theorem in \S\ref{subsec:proofs} along with our
main results in the local setting stated in \S\ref{subsec:local
  setting}.

\begin{rem}
  \label{rem:three}
  When $p=3$, the hypothesis that $\rhobar$ is generic implies that
  $e' = 1$.  Since the weight part of
  Serre's conjecture in the unramified case (i.e.~the Buzzard--Diamond--Jarvis
  conjecture) is already known in full \cite{blggU2,GLS2,GeeKisin},
  Theorem~\ref{thm:mainglobalresult} provides new information only
  when $p > 3$.
\end{rem}

\subsection{The local setting}
\label{subsec:local setting}
We will now revert to the setting of \S\ref{sec:weights-types}, where
$\rhobar: G_L \to \GL_2(k_E)$ is a reducible representation, written as
$$\rhobar \simeq  \left(\begin{array}{cc} \chi_2 & * \\ 0 & \chi_1 \end{array}\right);$$
moreover we assume $\rhobar$ is generic.

Suppose now that $\mu$ is a Serre weight in $W$ in the notation of \S\ref{subsec:partition}.
Recall that $W$ is a subset of $W_\expl(\rhobar^\ssimp)$ with complement of
cardinality at most $2$, and that $\mu = \mu(J,d)$ for some $(J,d)$ satisfying  (\ref{eq:Schein_weights}), where $J \subseteq \{0,\ldots,f-1\}$
and $d = (d_0,\ldots,d_{f-1})$ with $0\le d_i \le e'-1$ if $i\in J$, $1\le d_i \le e'$
if $i \not\in J$.

Now choose a lift $\tilde{\sigma}_i: L \into E$ of $\sigma_i$ for each $i \in \{0,\ldots,f-1\}$
and a subset $\tilde{J} \subseteq \{\sigma: L \into E\}$ such that
\begin{itemize}
\item $\tilde{\sigma}_i \in \tilde{J}$ if and only if $i \in J$, and
\item $\{\,\sigma\in \tilde{J}\,:\, \mbox{$\sigma$ is a lift of $\sigma_i$}\,\}$ has cardinality $e' - d_i$.
\end{itemize}
Choose also a crystalline character $\tilde{\chi}_1:G_L \to E^\times$ lifting $\chi_1$
whose Hodge--Tate module $V_1$ has $\sigma$-labelled weights:
\begin{itemize}
\item $1$, if $\sigma\not\in \tilde{J}$ and $\sigma\not\in\{\,\tilde{\sigma}_i\,:\,i = 0,\ldots,f-1\,\}$;
\item $0$, if $\sigma\in \tilde{J}$ and $\sigma\not\in\{\,\tilde{\sigma}_i\,:\,i = 0,\ldots,f-1\,\}$;
\item $m_i+n_i+1$, if $\sigma = \tilde{\sigma}_i \not \in J$;
\item $m_i$, if $\sigma = \tilde{\sigma}_i \in J$.
\end{itemize}
That such a crystalline character exists follows for example from 
Lubin--Tate theory, or from \cite[Prop.~B.3]{conradlifting}; 
moreover such a character
is unique up to an unramified twist with trivial reduction.
Similarly, let $\tilde{\chi}_2:G_L \to E^\times$ be a lift of $\chi_2$
whose Hodge--Tate module $V_2$ has $\sigma$-labelled weights:
\begin{itemize}
\item $0$, if $\sigma\not\in \tilde{J}$ and $\sigma\not\in\{\,\tilde{\sigma}_i\,:\,i = 0,\ldots,f-1\,\}$;
\item $1$, if $\sigma\in \tilde{J}$ and $\sigma\not\in\{\,\tilde{\sigma}_i\,:\,i = 0,\ldots,f-1\,\}$;
\item $m_i$, if $\sigma = \tilde{\sigma}_i \not \in J$;
\item $m_i+n_i+1$, if $\sigma = \tilde{\sigma}_i \in J$.
\end{itemize}
Note that $V_1 \oplus V_2$ is a Hodge--Tate module lifting $\mu$.

We let $L_{\cris,E}(\tilde{\chi}_1,\tilde{\chi}_2)$ denote the subspace 
of $\Ext^1_{E[G_L]}(\tilde{\chi}_1,\tilde{\chi}_2)$ corresponding to
the set of extensions which are crystalline.  We let
$L_{\cris,\cO_E}(\tilde{\chi}_1,\tilde{\chi}_2)$ denote the preimage of
$L_{\cris,E}(\tilde{\chi}_1,\tilde{\chi}_2)$ in
$\Ext^1_{\cO_E[G_L]}(\tilde{\chi}_1,\tilde{\chi}_2)$, and let
$L_{\cris,k_E}(\tilde{\chi}_1,\tilde{\chi}_2)$ denote the image of
$L_{\cris,\cO_E}(\tilde{\chi}_1,\tilde{\chi}_2)$ in
$\Ext^1_{k_E[G_L]}(\chi_1,\chi_2)$.  Thus 
$L_{\cris,k_E}(\tilde{\chi}_1,\tilde{\chi}_2)$ consists of the set of
extensions arising as reductions of crystalline representations of the form
$\left(\begin{array}{cc} \tilde{\chi}_2 & * \\ 0 & \tilde{\chi}_1 \end{array}\right)$.

Recall that we have defined a partition of $W$ into
subsets $W_a$ indexed by $f$-tuples $(a_0,a_1,\ldots,a_{f-1})$ with $0 \le a_i \le e'$
for all $i$ (Proposition~\ref{prop:partition}).  Our main result comparing
reductions of crystalline and potentially Barsotti--Tate extensions is the
following.
\begin{thm} \label{thm:flat_vs_crys}  
If $\mu \in W_a$, then 
$$L_{\cris,k_E}(\tilde{\chi}_1,\tilde{\chi}_2) = L(\chi_1,\chi_2,\tau_a).$$
\end{thm}
\begin{rem}\label{rem:choice-indepedence}
Note in particular that not only is $L_{\cris,k_E}(\tilde{\chi}_1,\tilde{\chi}_2)$
independent of the weight $\mu \in W_a$, but also of the various choices of lifts
$\tilde{\sigma}_i$, $\tilde{J}$, $\tilde{\chi}_1$ and $\tilde{\chi}_2$.
\end{rem}

Next we state the main result on the possible forms of $W_\expl(\rhobar)$.
Recall that the partition of $W$ into the $W_a$ extends to a partition of
$W_\expl(\rhobar^\ssimp)$ into subsets $W_a'$ defined in \S\ref{subsec:partition}.
To treat the case that $\rhobar$ is equivalent to a representation of the form 
$\chi_1 \otimes \left(\begin{array}{cc} \epsilonbar & * \\ 0 & 1 \end{array}\right)$,
recall that such a representation is {\em tr\`es ramifi\'ee} if the splitting field
of its projective image is {\em not} of the form $L(\alpha_1^{1/p},\ldots,\alpha_s^{1/p})$
for some $\alpha_1,\ldots,\alpha_s \in \cO_L^\times$.

\begin{thm} \label{thm:weight_shape}  We have 
$$W_\expl(\rhobar) =  \coprod_{a \le a^{\max}} W'_a$$
for some $a^{\max} \in A$ depending on $\rhobar$, unless
$\rhobar$ is tr\`es ramifi\'ee, in which case 
$W_\expl(\rhobar) = \{\mu_{m,n}\}$ where $\chi_1|_{I_L} = \prod_{i=0}^{f-1}\omega_i^{m_i}$
and $n = (p-1,\ldots,p-1)$.
\end{thm}

\begin{rem}\label{rem:weight_shape}
It will also be clear from the proof that given any pair of characters $\chi_1$, $\chi_2$
such that $\chi_1 \oplus \chi_2$ is generic, every element of $A$ arises as $a^{\max}$
for some peu ramifi\'ee extension of $\chi_1$ by $\chi_2$.   The theorem therefore
completely determines the possible values of $W_\expl(\rhobar)$ for generic
$\rhobar$.  As indicated in the footnote after Definition~\ref{defn:generic},
we expect a similar description to be valid under hypotheses weaker than genericity,
but not in full generality; see Section~5.
\end{rem}

\subsection{The proofs}
\label{subsec:proofs}
In this section we will prove Theorems~\ref{thm:mainglobalresult}, 
\ref{thm:flat_vs_crys} and~\ref{thm:weight_shape}, but first we note
the following lemma.
\begin{lem}\label{lem:pBT_lifts}
Suppose that $\rhobar:G_F \to \GL_2(k_E)$ is as in 
Theorem~\ref{thm:GeeKisin} 
with $\rhobar|_{G_L}\simeq  \left(\begin{array}{cc} \chi_2 & * \\ 0 &
    \chi_1 \end{array}\right)$, and that $\tau$ is a principal series type. 
If $W^v_\modular(\rhobar)$ contains a Jordan--H\"older factor of
$\overline{\theta}_{\tau}$,
then the extension defined by
$\rhobar|_{G_L}$ is in $L(\chi_1,\chi_2,\tau)$.
\end{lem}
\begin{proof}  By \cite[Prop.~2.10]{bdj} (stated there only for $r=1$ and
$p$ unramified in $F$, but the proof carries over to our setting; see
also the proof of \cite[Lem.~3.4]{geesavitttotallyramified}), 
we have (replacing $E$ by an extension $E'$ if necessary) that $\rhobar \simeq \rhobar_\pi$ for some cuspidal holomorphic automorphic representation $\pi$ of $\GL_2(\A_F)$ such that $\pi_\infty$
is holomorphic of weight $(2,\ldots,2)$ with trivial central character and $\pi_v$ has
$K$-type $\theta_\tau$.  (Note that our normalizations for the local
Langlands correspondence differ from those of \cite{bdj};  in our case
$I(\psi_1\otimes\psi_2)$ corresponds to
$|\cdot|^{1/2}(\psi_1\oplus\psi_2)$.)

 Local--global compatibility at $v$ of the Langlands 
correspondence (see the Corollary in the introduction to \cite{MR2373358}) 
 therefore implies that
$\rho_\pi|_{G_L}$ is potentially Barsotti--Tate with associated Weil--Deligne
representation of type $\tau$.  (Note that when $\tau$ is scalar, by
definition $\theta_\tau$ is not a twist of the Steinberg representation.)   It follows from~\cite[Cor.~5.2]{MR2822861} that $\rhobar|_{G_L}$ is the generic fibre
of a model of type $\tau$ 
in the sense of 
Definition~\ref{defn:model-of-type-tau},  and hence that its associated
extension class is in $L(\chi_1,\chi_2,\tau)$.  Note furthermore that
replacing $E$ by $E'$ has the effect of replacing $L(\chi_1,\chi_2,\tau)$
by $L(\chi_1,\chi_2,\tau)\otimes_{k_E}k_{E'}$ (for example as an
application of 
Theorem~\ref{thm:extensions}), so that the conclusion holds
without having extended scalars.
\end{proof}

\noindent{\em Proof of Theorem~\ref{thm:flat_vs_crys}.}
Since $\tilde{\chi}_1$ and $\tilde{\chi}_2$ are distinct characters,
it follows from \cite[Prop.~1.24(2)]{nekovar} that
$$\dim_E L_{\cris,E}(\tilde{\chi}_1,\tilde{\chi_2}) = \dim_E(V/\Fil^0(V)),$$
where $V$ is the Hodge--Tate module $\Hom_E(V_1,V_2)$. 
 Note that $\dim_E(V/\Fil^0(V))$ is simply the number
of $\sigma:L \into E$  such that the $\sigma$-labelled Hodge--Tate weight of
$V_2$ is greater than that of $V_1$, which is the case if and only if 
$\sigma \in \tilde{J}$.  Therefore
$$\dim_E L_{\cris,E}(\tilde{\chi}_1,\tilde{\chi_2}) = | \tilde{J}|
  = \sum_{i=0}^{f-1} (e' - d_i).$$

The genericity hypothesis implies that $\chi_1 \neq \chi_2$, from which
it follows that $\Ext^1_{\cO_E[G_L]}(\tilde{\chi}_1,\tilde{\chi}_2)$ is torsion-free.
Therefore $L_{\cris,\cO_E}(\tilde{\chi}_1,\tilde{\chi}_2)$ is torsion-free over
$\cO_E$ of rank $\dim_E L_{\cris,E}(\tilde{\chi}_1,\tilde{\chi_2})$, and it follows that
$$\dim_{k_E} L_{\cris,k_E}(\tilde{\chi}_1,\tilde{\chi}_2) = \sum_{i=0}^{f-1} (e' - d_i).$$
By Proposition~\ref{prop:partition}(3) and Theorem~\ref{thm:intersections-redux}(1),
this is the same as the dimension of $L(\chi_1,\chi_2,\tau_a)$, so it suffices to
prove that
$$L_{\cris,k_E}(\chi_1,\chi_2) \subseteq L(\chi_1,\chi_2,\tau_a).$$
Moreover since these subspaces of $\Ext^1_{k_E[G_L]}(\chi_1,\chi_2)$
are well-behaved under extension of scalars, we may enlarge $E$ in
order to prove the inclusion.

Suppose now that we are given a representation $\overline{\varrho}:G_L \to \GL_2(k_E)$
giving rise to an extension class in $L_{\cris,k_E}(\chi_1,\chi_2)$.
By \cite[Cor.~A.3]{GeeKisin} 
we have that $\overline{\varrho} \simeq \rhobar|_{G_L}$
for some totally real field $F$, representation $\rhobar:G_F \to \GL_2(k_E)$ and embedding
$F \into L$ such that Theorem~\ref{thm:GeeKisin} applies (enlarging $E$ if necessary).
Since  $\mu \in W_\expl(\overline{\varrho})$, Theorem~\ref{thm:GeeKisin} implies that
$\mu \in W^v_\modular(\rhobar)$, and hence Lemma~\ref{lem:pBT_lifts} 
implies that the extension defined by $\overline{\varrho}$ is in $L(\chi_1,\chi_2,\tau_a)$.
\hfill$\square$

\bigskip

\noindent{\em Proof of Theorem~\ref{thm:weight_shape}.}
From Theorem~\ref{thm:flat_vs_crys}, Remark~\ref{rem:choice-indepedence}, and the definitions of $W_\expl(\rhobar)$ and
 $L_{\cris,k_E}(\tilde{\chi}_1,\tilde{\chi}_2)$, we see that if $\mu \in W_a$, then 
 $\mu \in W_\expl(\rhobar)$ if and only if the extension
class associated to $\rhobar$ is in $L(\chi_1,\chi_2,\tau_a)$.
Let $A_{\rhobar}$ denote the set of $a\in A$ for which this holds, so that
$$W_\expl(\rhobar) \cap W = \coprod_{a\in A_{\rhobar}} W_a.$$

By Theorem~\ref{thm:intersections-redux}(1), we have
$\dim_{k_E}L(\chi_1,\chi_2,\tau_{(0,\ldots,0)})  = e'f$.
If $\chi_2 \neq \chi_1\epsilonbar$, then this is the same as the dimension of
$$\Ext^1_{k_E[G_L]}(\chi_1,\chi_2) \simeq  H^1(G_L,\chi_1^{-1}\chi_2),$$
so we have that $(0,\ldots,0) \in A_{\rhobar}$, and in particular $A_{\rhobar}$
is nonempty.  In the case that $\chi_2 = \chi_1\epsilonbar$, we have the isomorphism
$$\Ext^1_{k_E[G_L]}(\chi_1,\chi_2) \simeq L^\times/(L^\times)^p \otimes k_E$$
of Kummer theory.  Note that the genericity hypothesis implies that
$\zeta_p \not\in L$, so these spaces have dimension $e'f + 1$.
The subspace $\cO_L^\times/(\cO_L^\times)^p \otimes k_E$
has dimension $e'f$, and the corresponding classes arise as generic fibres
of models of type $\tau_{(0,\ldots,0)} = (\chi_1 \oplus
\chi_1)|_{I_L}$.  To see this,  twist by $\chi_1^{-1}$ to reduce to the case
where $\tau_{(0,\ldots,0)}$ is trivial, and then apply
 \cite[Prop.~8.2]{Edixhoven} (or
more precisely the analogous statement with~$L$ in place of $\Qp$,
which follows by the same proof).
Therefore $L(\chi_1,\chi_2,\tau_{(0,\ldots,0)})$ corresponds to 
$\cO_L^\times/(\cO_L^\times)^p \otimes k_E$, and 
$(0,\ldots,0) \in A_{\rhobar}$ if and only if $\rhobar$ is not
tr\`es ramifi\'ee.

Suppose now that $\rhobar$ is not tr\`es ramifi\'ee.    In particular $A_{\rhobar}$
is nonempty and Theorem~\ref{thm:intersections-redux}(2) implies that
$$\bigcap_{a\in A_{\rhobar}}  L(\chi_1,\chi_2,\tau_a) 
 = L(\chi_1,\chi_2,\tau_{a^{\max}})$$
 where $a_i^{\max} = \max_{a\in A_{\rhobar}} \{a_i\}$.  Moreover
$a \in A_{\rhobar}$ if and only $a \le a^{\max}$, so
$$W_\expl(\rhobar) \cap W = \coprod_{a \le a^{\max}} W_a.$$
On the other hand if $\rhobar$ is tr\`es ramifi\'ee, then we see that
$W_\expl(\rhobar) \cap W = \varnothing$.

To complete the proof of the theorem, we must treat the two possible
additional weights $\mu'(J,d)$ arising when 
$\chi_1^{-1}\chi_2|_{I_L} = \epsilonbar|^{\pm 1}_{I_L}$.

Note that the dimension calculations at the
beginning of the proof of Theorem~\ref{thm:flat_vs_crys} apply equally
with $n = (0,\ldots,0)$ replaced by $n = (p-1,\ldots,p-1)$. 
In the case
$\chi_1|_{I_L} = \chi_2\epsilonbar|_{I_L}$, $J = \varnothing$ and 
$d = (e',\ldots,e')$, this gives  $L_{\cris,k_E}(\tilde{\chi}_1,\tilde{\chi}_2) = \{0\}$,
so that 
$$\mu'(J,d) \in W_\expl(\rhobar)\quad \Leftrightarrow \quad\mbox{$\rhobar$ splits}\quad
\Leftrightarrow\quad \mu(J,d) \in W_\expl(\rhobar).$$
In the case
$\chi_2|_{I_L} = \chi_1\epsilonbar|_{I_L}$, $J = \{0,\ldots,f-1\}$ and $d = (0,\ldots,0)$,
we find that
$$\dim_{k_E}L_{\cris,k_E}(\tilde{\chi}_1,\tilde{\chi}_2) = e'f,$$
and we must show that $\mu'(J,d) \in W_\expl(\rhobar)$.
If $\chi_1 \neq \chi_2\epsilonbar$, then this holds since
$$L_{\cris,k_E}(\tilde{\chi}_1,\tilde{\chi}_2) = \Ext^1_{k_E[G_L]}(\chi_1,\chi_2).$$
If $\chi_1 = \chi_2\epsilonbar$, then we must show that every class in 
$$ \Ext^1_{k_E[G_L]}(\chi_1,\chi_2) \simeq H^1(G_L,k_E(\epsilonbar))$$
is in the codimension one subspace $L_{\cris,k_E}(\tilde{\chi}_1,\tilde{\chi}_2)$
for some choice of lifts $\tilde{\chi}_1$, $\tilde{\chi_2}$ as in \S\ref{subsec:local setting}
with $n = (p-1,\ldots,p-1)$, enlarging $E$ if necessary.  This follows
by exactly the same proof as that of
\cite[Prop.~5.2.9]{GLS11}.
\hfill$\square$

\bigskip

\noindent{\em Proof of Theorem~\ref{thm:mainglobalresult}.}
We must show that if $\mu \in W' \cap W^v_\modular(\rhobar)$, then 
$\mu \in W_\expl(\rhobar|_{G_L})$.  Suppose first that $\mu \in W_a$
for some $a \in A$.  By Lemma~\ref{lem:pBT_lifts}, we have that the
extension class associated to $\rhobar|_{G_L}$ is in $L(\chi_1,\chi_2,\tau_a)$,
hence in $L_{\cris,k_E}(\tilde{\chi}_1,\tilde{\chi}_2)$ by Theorem~\ref{thm:flat_vs_crys},
and therefore in $W_\expl(\rhobar|_{G_L})$.

Now we must deal with the two exceptional weights.  If
$\chi_2|_{I_L} = \chi_1\epsilonbar|_{I_L}$, then Theorem~\ref{thm:weight_shape}
 implies that $\mu'(J,d) \in W_\expl(\rhobar|_{G_L})$, where
 $J = \{0,\ldots,f-1\}$ and $d = (0,\ldots,0)$.
Finally suppose that $\chi_1|_{I_L} = \chi_2\epsilonbar|_{I_L}$
and that 
$\mu'(J,d) \in W^v_\modular(\rhobar)$, where
 $J = \varnothing$ and $d = (p-1,\ldots,p-1)$.
 In this case the same argument as in the proof of Lemma~\ref{lem:pBT_lifts}
 shows that $\rhobar \simeq \rhobar_\pi$ for some 
cuspidal holomorphic automorphic representation $\pi$ of $\GL_2(\A_F)$ such that $\pi_\infty$
is holomorphic of weight $(2,\ldots,2)$ with trivial central character and $\psi\otimes\pi_v$
has a vector invariant under $U_0(v)$, where $\psi = [\chi_2]^{-1}\circ\det$ and $[\chi_2]$
denotes the Teichm\"uller lift of $\chi_2$.
Therefore $\psi\otimes\pi_v$ is either an unramified principal series, or an unramified
twist of the Steinberg representation.  If $\psi\otimes\pi_v$ is unramified, then
in fact $\rhobar$ is modular of weight $\mu(J,d)$, so it follows
from the cases already proved that $\mu(J,d) \in W_\expl(\rhobar|_{G_L})$, and
hence $\mu'(J,d) \in W_\expl(\rhobar|_{G_L})$ by Theorem~\ref{thm:weight_shape}.  
(In fact we see from the proof of Theorem~\ref{thm:weight_shape} that in this
case $\rhobar|_{G_L}$ is split.)  If $\psi\otimes\pi_v$ is an unramified twist of
the Steinberg representation, then local-global compatibility at $v$ gives
that $\rho_\pi|_{G_L}$ is an unramified twist of a representation of the form
$[\chi_2]\otimes \left(\begin{array}{cc}\epsilon  & * \\ 0 & 1 \end{array}\right)$.
Since $\chi_1|_{I_L} \neq \chi_2|_{I_L}$, it follows that $\rhobar|_{G_L}$
is split, and hence that $\mu'(J,d) \in W_\expl(\rhobar|_{G_L})$
in this case as well.
\hfill$\square$

\section{A remark on the genericity hypothesis}
\label{sec:counterexample}

In this section we show that the genericity hypothesis 
on $\rhobar$ is, in general, necessary in order for our arguments in
the proof of Theorem~\ref{ithm:main-global} to
go through.   That is, we give an example of a field $L$, characters $\chi_1$,
$\chi_2 : G_L \to \Fpbar^{\times}$, and a weight $\mu$ such that the
subset  $L_{\textrm{cris}} \subseteq H^1(G_{L},\Fpbar(\chi_2
\chi_1^{-1}))$ corresponding to the representations $\rhobar$ with $\mu \in
W_{\expl}(\rhobar)$ is not equal to the space
$L(\chi_1,\chi_2,\tau)$ for any principal series type $\tau$ such that
$\mu$ is  a Jordan--H\"older constituent of $\overline{\theta}_\tau$.

Let $L = \Q_{p^2}$ be the unramified quadratic extension of $\Qp$, so
that $f = 2$ and $e' = 1$.  Take
$\chi_1$ to be trivial, and $\chi_2 = \chi$ to be  any extension to $G_L$ of
$\omega_0^{p-1} \omega_1^b$, where $b \in [1,p-2]$ is an integer.
Observe that the weight $\mu_{m,n}$ with
$$ m = (p-1,b-1),\qquad n = (p-1,p-b-1) $$
lies in $W_\expl(\rhobar^{\textrm{ss}})$.  One checks that $J = \{0\}$ is the only
subset $J \subseteq \{0,1\}$  such that $$\prod_{i \in J} \omega_i^{m_i +
  d_i} \prod_{i \not\in J} \omega_i^{m_i+n_i+d_i} = 1$$ with $d_i = 0$
if $i \in J$ and $d_i = 1$ otherwise.  It follows as in the proof of Theorem~\ref{thm:flat_vs_crys} that $\dim_{k_E}
L_{\textrm{cris},k_E}(\tilde{\chi}_1,\tilde{\chi}_2) = (1-1) + (1-0) = 1$,
where $\tilde{\chi}_1$ and $\tilde{\chi}_2$ are defined as in \S\ref{sec:mainsetting}.
By \cite[Rem.~7.13]{MR2776609}, the spaces $L_{\textrm{cris},k_E}(\tilde{\chi}_1,\tilde{\chi}_2)$ are independent of the choice of $\tilde{\chi}_1$, $\tilde{\chi}_2$, so in fact
$\dim_{k_E}L_{\textrm{cris}} = 1$.

On the other hand, one checks (e.g. from
\cite[Prop.~1.1]{DiamondDurham}) that there is exactly one principal
series type $\tau$ such that $\mu$ is a Jordan--H\"older constituent
of $\overline{\theta}_{\tau}$, namely 
$$ \tau \simeq \omega_0^{p-2} \omega_1^{p-1} \oplus \omega_0^{p-1}
\omega_1^{b-1}.$$
The weight $\mu_{m',n'}$ with $m' = (0,0)$ and $n'=(p-2,b-1)$ is also
a Jordan-H\"older constituent of $\overline{\theta}_\tau$ as well as
an element of $W_\expl(\rhobar^{\textrm{ss}})$;  moreover the 
space $L_{\textrm{cris},k_E}(\tilde{\chi}_1',\tilde{\chi}_2')$
corresponding to $\mu_{m',n'}$ has dimension $2$,
hence is equal to $\Ext^1_{k_E[G_L]}(\chi_1,\chi_2)$.
As in the proof of Theorem~\ref{thm:flat_vs_crys}, we have that
$L_{\textrm{cris},k_E}(\tilde{\chi}_1',\tilde{\chi}_2') \subseteq L(\chi_1,\chi_2,\tau)$,
so that  $L(\chi_1,\chi_2,\tau) = \Ext^1_{k_E[G_L]}(\chi_1,\chi_2)$ properly
contains $L_{\textrm{cris}}$.

We remark however that in the case $f=2$ and $e'=1$,
 Corollary~5.13 and Theorem~7.12 of \cite{MR2776609} yield
a partition of $W_{\expl}(\rhobar^{\textrm{ss}})$ into subsets $W_a'$ for
$a \in \{0,1\}^2$ such that Theorem~\ref{thm:weight_shape} holds even without
the genericity hypothesis.  Indeed in the example above (with $\chi_1^{-1}\chi_2|_{I_L} = 
\omega_0^{p-1}\omega_1^b$ for some $b\in [1,p-2]$), one even has that
each $W_a'$ is a singleton exactly as in the generic case.
On the other hand if $\chi_1^{-1}\chi_2|_{I_L} = \omega_1^b$
for some $b\in[1,p-1]$, then one of the two subsets $W_a'$ with $a_0 + a_1 = 1$
must be empty, while the other has cardinality two.

\bibliographystyle{plain} 
\bibliography{refs} 

\end{document}